\documentclass[a4paper, 12pt]{article}

\usepackage{subfiles}
\usepackage[toc,page]{appendix}
\usepackage[bottom]{footmisc}
\usepackage{microtype}
\usepackage{tocbibind}
\usepackage{stmaryrd}   
\usepackage{amsmath}
\usepackage{verbatim}
\usepackage{mathrsfs}
\usepackage{mathtools}
\usepackage{amsthm} 
\usepackage{amssymb} 
\usepackage{faktor}
\usepackage{array}
\usepackage[colorlinks={true},linkcolor={blue},citecolor={blue}, urlcolor={blue} ]{hyperref}
\usepackage{braket}
\usepackage{transparent}
\usepackage{graphicx}
\usepackage{color}
\usepackage{rotating}
\usepackage{tikz}
\usepackage{tikz-cd}
\usepackage{authblk}
\usepackage[margin=2.5 cm]{geometry}
\usepackage{bm}
\usepackage{cite}

\usepackage[T1]{fontenc} 

\DeclareMathOperator{\udeg}{%
  \vphantom{\mathrm{\deg}}%
  \underline{\smash[b]{\mathrm{\deg}}}
}

\newcommand{\uu}[1]{\underline{\underline{#1}}}
\newcommand{\uv}{\underline{v}}

\DeclareMathOperator{\uV}{\underline{\V}}
\DeclareMathOperator{\defeq}{\overset{\textnormal{\tiny{def}}}{=}}
\DeclareMathOperator\mcO{\mathcal{O}}    
\DeclareMathOperator{\N}{\mathbb{N}}
\DeclareMathOperator{\Z}{\mathbb{Z}}
\DeclareMathOperator{\Q}{\mathbb{Q}}
\DeclareMathOperator{\R}{\mathbb{R}}
\DeclareMathOperator{\C}{\mathbb{C}}
\DeclareMathOperator{\F}{\mathbb{F}}
\DeclareMathOperator{\V}{\mathbb{V}}

\DeclareMathOperator{\Primes}{\mathfrak{Primes}}

\DeclareMathOperator{\rad}{\textnormal{rad}}

\DeclareMathOperator{\good}{\textnormal{good}}
\DeclareMathOperator{\bad}{\textnormal{bad}}
\DeclareMathOperator{\multi}{\textnormal{multi}}
\DeclareMathOperator{\gen}{\textnormal{gen}}
\DeclareMathOperator{\non}{\textnormal{non}}
\DeclareMathOperator{\arc}{\textnormal{arc}}
\DeclareMathOperator{\Vol}{\textnormal{Vol}}
\DeclareMathOperator{\tmod}{\textnormal{mod}}
\DeclareMathOperator{\Gal}{\textnormal{Gal}}

\DeclareMathOperator{\GL}{\textnormal{GL}}
\DeclareMathOperator{\SL}{\textnormal{SL}}

\DeclareMathOperator{\ADiv}{\textnormal{ADiv}}

\numberwithin{equation}{section}

\theoremstyle{plain} %
\newtheorem{thm}{Theorem}[section]
\newtheorem{prop}[thm]{Proposition}
\newtheorem{lem}[thm]{Lemma}
\newtheorem{cor}[thm]{Corollary}
\newtheorem{definition}[thm]{Definition}

\theoremstyle{definition} %

\newtheorem{tiny-remark}{Remark}[thm]

\theoremstyle{plain} %
\newtheorem*{thm2}{Theorem}
\newtheorem*{cor2}{Corollary}
\newtheorem*{result-A1}{Theorem A1}
\newtheorem*{result-A2}{Corollary A2}
\newtheorem*{result-A3}{Theorem A3}
\newtheorem*{result-A4}{Corollary A4}
\newtheorem*{result-B1}{Theorem B1}
\newtheorem*{result-B2}{Corollary B2}
\newtheorem*{result-B3}{Corollary B3}

\makeatletter

\title{The inter-universal Teichm\"uller theory and new Diophantine results over the rational numbers. I}

\author{Zhong-Peng Zhou}
\date{}

\newcommand{\Addresses}{{
\bigskip
(Zhong-Peng Zhou) \textsc{Institute for Theoretical Sciences, Westlake University, 
No. 600 Dunyu Road,
Xihu District, Hangzhou,
Zhejiang, 310030, P.R. China}
\par\nopagebreak
\textit{E-mail address}: 
\texttt{zhouzhongpeng@westlake.edu.cn}
}}

\begin{document}

\maketitle

\renewcommand{\thefootnote}{}
\footnotetext{Submitted for publication on September 11, 2024.}

\begin{abstract}

By applying inter-universal Teichm\"uller theory and its slight modification over the rational number field, we prove new Diophantine results towards effective abc inequalities and the generalized Fermat equations.

For coprime integers \(a, b, c\) satisfying \(a + b = c\) and \(\log(|abc|) \geq 700\), we prove
\[
\log|abc| \leq 3\log\mathrm{rad}(abc) + 8\sqrt{\log|abc| \cdot \log\log|abc|}.
\]
This implies for any \(0 < \epsilon \leq \frac{1}{10}\),
\[
|abc| \leq \max\left\{\exp\left(400 \cdot \epsilon^{-2} \cdot \log(\epsilon^{-1})\right),\ \mathrm{rad}(abc)^{3+3\epsilon}\right\},
\]
reducing the constant in effective abc bounds from \(1.7 \cdot 10^{30}\) (Mochizuki-Fesenko-Hoshi-Minamide-Porowski) to \(400\).

For positive primitive solutions \((x, y, z)\) to the generalized Fermat equation  \(x^r + y^s = z^t\) (\(r, s, t \geq 3\)), define \(h = \log(x^r y^s z^t)\). We prove explicit bounds:
\begin{gather*}
h \leq 573\ \ (r, s, t \geq 8); \;
h \leq 907\ \ (r, s, t \geq 5); \;
h \leq 2283\ \ (r, s, t \geq 4); \\
h \leq 14750\ \ (\min\{r, s\} \geq 4\ \text{or}\ t \geq 4); \;
h \leq 24626\ \ (r, s, t \geq 3).
\end{gather*}
These imply Fermat's Last Theorem (FLT) holds unconditionally for prime exponents \(\geq 11\). Combined with classical results for FLT with exponents \(3, 4, 5, 7\), this yields a new alternative proof of FLT. Computational verification confirms no non-trivial primitive solution exists when \(r, s, t \geq 20\) or \((r, s, t)\) is a permutation of \((3, 3, n)\) (\(n \geq 3\)).

\end{abstract}

\makeatletter
\@starttoc{toc}
\makeatother

\section*{Introduction}
\addcontentsline{toc}{section}{Introduction}

In the recent paper \cite{ExpEst}, Shinichi Mochizuki, Ivan B. Fesenko, Yuichiro Hoshi, Arata Minamide, and Wojciech Porowski established a $\mu_6$-version of Mochizuki's inter-universal Teichm\"uller theory (IUT theory) [cf. \cite{IUTchI,IUTchII,IUTchIII,IUTchIV}],
proving the following diophantine results [cf. \cite{ExpEst}, Theorem B and Corollary C]:

\begin{thm2}\textbf{(Effective version of a conjecture of Szpiro)}
     Let $a, b, c$ be non-zero coprime integers such that $a+b+c = 0$; 
     $\epsilon$ a positive real number $\le 1$. 
     Then we have
	\begin{align*}
	    |abc| \le 2^4 \cdot \max\{ \exp(1.7 \cdot 10^{30} \cdot \epsilon^{-166/81}, \rad(abc)^{3+3\epsilon}
	     \}.
	\end{align*}
\end{thm2}

 \begin{cor2}\textbf{(Application to ``Fermat's Last Theorem'')}
      Let $$p > 1.615 \cdot 10^{14}$$ be a prime number. 
      Then there does not exist any triple $(x,y,z)$ of positive integers that satisfies the Fermat equation $$x^p + y^p = z^p.$$
 \end{cor2}

In the present paper, we slightly modify some of the methods in [38, 39], in the special case of rational numbers, to estimate the log-volumes of  [$\mu_6$-]initial $\Theta$-data more precisely. Then by focusing on the Frey-Hellegouarch curves, we prove new effective versions of abc inequalities for rational numbers:

\begin{result-A1}
Let $(a,b,c)$ be a triple of non-zero coprime integers such that $a+b=c$.

(i) Suppose that $\log(|abc|)\ge 700$, then we have
\begin{align*}
    \log(|abc|) \le 3\log\rad(abc) + 8 \sqrt{\log(|abc|)\cdot \log\log(|abc|)}.
\end{align*} 

(ii) Suppose that $\log(|abc|)\ge 3\cdot 10^{13}$, then we have
\begin{align*}
    \log(|abc|) \le 3\log\rad(abc) + 3&\sqrt{\log(|abc|)\cdot \log\log(|abc|)}
    \\ &\cdot (1 + \frac{6}{\sqrt{\log\log(|abc|)}} + \frac{15}{\log\log(|abc|)}).
\end{align*} 
\end{result-A1}

\begin{result-A2} 
Let $(a,b,c)$ be a triple of non-zero coprime integers such that $a+b=c$;
$\epsilon$ be a positive real number $\le \frac{1}{10}$. Then we have 
\begin{align*}
    |abc| \le \max\{
    \exp\left(400 \cdot \epsilon^{-2} \cdot \log(\epsilon^{-1} ) \right),
    \rad(abc)^{3+3\epsilon} \}.
\end{align*}
\end{result-A2}

We will also conduct a preliminary study on the generalized Fermat equations.
Let $r,s,t\ge 2$ be positive integers. The equation
\begin{equation*}
x^r + y^s = z^t, \;\text{with}\; x, y, z\in\Z, \; xyz\neq 0
\end{equation*}
is known as the generalized Fermat equation with signature $(r, s, t)$.
A solution $(x, y, z)$ is called primitive if $\gcd(x, y, z)=1$, and positive if $x, y, z \ge 1$.

\begin{result-B1}\label{3-prop:upper bound of solutions}
Let $r,s,t\ge 3$ be positive integers.
Write
\begin{align*}
S(r,s,t) \defeq \{ (x,y,z)\in \Z_{\ge 1}^3  :  x^r + y^s = z^t,\, \gcd(x,y,z) = 1 \}
\end{align*}
for the finite set of positive primitive solutions to the generalized Fermat equation with signature $(r,s,t)$. 
Define
\begin{align*}
f(r,s,t) \defeq \sup\limits_{(x,y,z) \in S} \log(x^r y^s z^t) \in  \R_{\ge 0},
\end{align*}
where we shall write $f(r,s,t) = 0$ if $S(r,s,t)$ is an empty set.

Then we have $f(r,s,t) \le 573$ for $r,s,t\ge 8$;
$f(r,s,t) \le 907$ for $r,s,t\ge 5$;
$f(r,s,t) \le 2283$ for $r,s,t\ge 4$;
$f(r,s,t) \le 14750$ for $\min\{r,s\} \ge  4$ or $t\ge 4$;
and $f(r,s,t) \le 24626$ for $r,s,t\ge 3$.
\end{result-B1}

\begin{result-B2} 
Let $$p\ge 11$$ be a prime number. 
Then there does not exist any triple $(x, y, z)$ of positive integers that satisfies the Fermat equation
$$x^p + y^p = z^p.$$
\end{result-B2}

Corollary B2, combined with the classical results about the non-existence of positive integer solutions $(x,y,z)$ of the Fermat equation 
$$x^n + y^n = z^n$$ 
for $n=3$ (Euler), $n=4$ (Fermat), $n=5$ (Dirichlet and Legendre) and $n=7$ (Gabriel Lamé), yields an \textbf{unconditional new alternative proof}  [i.e., to the proof of \cite{Wiles} and a certain simplification of its second proof in [39]] of \textbf{Fermat's Last Theorem}.

\begin{result-B3} 
Let $r,s,t$ be positive integers such that $r,s,t \ge 20$, or $(r,s,t)$ is a permutation of $(3,3,n)$ for $n\ge 3$.
Then there does not exist any triple $(x, y, z)$ of non-zero coprime integers that satisfies the generalized Fermat equation
$$x^r + y^s = z^t.$$
\end{result-B3}

In the subsequent paper, we will introduce new insights to further explore the generalized Fermat equations. We will establish new upper bounds for non-trivial primitive solutions and demonstrate the non-existence of such solution for an expanded set of signatures.

The present paper is organized as follows.
In Section 1, we provide an estimation for the log-volume of $\mu_6$-initial $\Theta$-data [cf. Proposition \ref{1-prop:Estimates the log-volume of initial Theta-data}],
which is a slightly modified version of \cite{IUTchIV}, Theorem 1.10 and \cite{ExpEst}, Theorem 5.1.
In Section 2, for any triple of non-zero coprime integers $(a, b, c)$ such that $a+b=c$,
we first construct the $\mu_6$-initial $\Theta$-data associated to the Frey-Hellegouarch curve $y^2 = x(x-a)(x+b).$
Then by applying the estimation in Section 1 to these $\mu_6$-initial $\Theta$-data, 
we obtain ``partial abc inequalities'' associated to $(a,b,c)$.
Section 3 is devoted to prove several effective abc inequalities.
Finally, in Section 4, we study on the upper bounds for positive primitive solutions to the generalized Fermat equations.

\section*{Acknowledgements}

The results of the present paper were obtained by its author when studying on his own initiative the relevant research papers and then working alone on this paper in Beijing in 2023-2024. Sharing the first draft of the paper with Shinichi Mochizuki and Ivan Fesenko led to their encouragement which is gratefully acknowledged. The author is also grateful to Ivan Fesenko for discussions of some topics of the present paper during his short visits to Westlake University.

\setcounter{section}{-1}
\section{Notations and Conventions}

\noindent\textbf{General Notations:}

For a real number $\lambda$, we shall write $\lfloor\lambda\rfloor$ for the largest integer $n\le \lambda$, and write $\lceil\lambda\rceil$ for the smallest integer $n\ge \lambda$.
For a finite set $I$, we shall write $|I|$ for the the cardinality of $I$. 
For a finite field extension $K/L$, we shall write $[K:L]$ for the extension degree $\dim_L K$. 
For a field $K$ or a ring $R$, we shall write $K^{\times}$ or $R^{\times}$ for the group of invertible elements in $K$ or $R$. 

Let $n$ be a non-zero integer. 
The radical of $n$, denoted by $\rad(n)$, 
is the product of the distinct prime factors of $n$.
For any integer $k\ge 2$, the $k$-radical of $n$, denoted by $\rad_k(n)$,
is the product of the distinct prime factors $p$ of $n$, 
where the $p$-adic valuation $v_p(n)$ of $n$ is not divided by $k$.
Hence we have $$\rad(n)=\prod_{p\mid n} p, \quad \rad_k(n)=\prod_{k\nmid v_p(n)}p. $$

Let $p$ be a prime number.
We shall write $\Q$ (respectively, $\R$; $\C$; $\Q_p$; $\F_p$) for the field of rational numbers (respectively, real numbers; complex numbers; $p$-adic rational numbers; the finite field with $p$ elements), 
and write $\Z$ (respectively, $\Z_p$) for the ring of rational integers (respectively, the ring of $p$-adic integers).

We shall call finite extension of $\Q$ number field, call finite extension of $\Q_p$ $p$-adic local field, call $\R$ real archimedean local field, and call $\C$ complex archimedean local field. 

We shall write $\Primes$ for the set of prime numbers. 
For each prime number $l\ge 3$, we shall put $l^{\divideontimes} \defeq \frac{l-1}{2}$.

\hfill \break
\noindent\textbf{Number Fields:}

Let $L$ be an arbitrary number field. 
We shall use $\V(L)$ to denote the set of valuations of $L$, 
and use $\V(L)^{\non}$ (respectively, $\V(L)^{\arc}$) to denote the set of nonarchimedean (respectively, archimedean) valuations of $L$. 
Here, all the valuations are assumed to be normalized.

For the case $L = \Q$, we shall put $\V_{\Q} \defeq \V(\Q)$, 
and put $\V_{\Q}^{\square} \defeq \V(\Q)^{\square}$ for $\square\in\{\non,\arc\}.$ 
Then we have $\V_{\Q}^{\non} = \{v_p: p\in \Primes\}$ and $\V_{\Q}^{\arc} = \{v_{\R}\}$, 
where $v_p$ is the $p$-adic valuation on $\Q$, 
and $v_{\R}$ is the unique archimedean valuation on $\Q$. 
For a prime number $p$ and a rational number $x$, 
if $v_p(x)\ge 1$, then we shall say $p$ divides $x$, and write it as $p\mid x$;
if $v_p(x)\le 0$, then we shall say $p$ does not divide $x$, and write it as $p\nmid x$.

Let $K/L$ be a finite extension of number fields. 
For each $v\in \V(L)$, we shall use $\V(K)_v$ to denote the [finite] set of valuations of $K$ lying over $v$. 
Then we have $$\sum_{\uv \in \V(K)_v} [K_{\uv}:L_v] = [K:L].$$

An [$\R$-]arithmetic divisor on $L$ is a finite formal sum $\mathfrak{a} = \sum_{v\in \V(L)} c_v \cdot v$, where $c_v\in\R$. 
We shall use $\ADiv_{\R}(L)$ to denote the abelian group of [$\R$-]arithmetic divisors on $L$. The degree of $\mathfrak{a}$ is defined by the formula 
\begin{align*}
	\deg_L(\mathfrak{a}) \defeq \sum_{v\in \V(L)^{\non}}c_v\cdot \log\big(|\kappa(L_v)|\big) + \sum_{v\in \V(L)^{\arc}} c_v,
\end{align*}
where $\kappa(L_v)$ is the [finite] residue field of $L_v$. 
The normalized degree of $\mathfrak{a}$ is defined by the formula
$$\udeg(\mathfrak{a}) \defeq \frac{1}{[L:\Q]}\deg_L(\mathfrak{a}). $$ 
Thus for any finite extension $K$ of $L$, we have $\udeg(\mathfrak{a}|_K ) = \udeg(\mathfrak{a}),$ where we write $\mathfrak{a}|_K$ for the the pull-back [defined in the evident fashion] of $\mathfrak{a}$ to an arithmetic divisor on $K$.

\hfill \break
\noindent\textbf{Local Fields:}

For an archimedean local field $k$ [i.e. $\R$ or $\C$] with absolute value $| \cdot |_k$, we shall write
$$\mcO_{k} \defeq \{x:x\in k, |x|_k \le 1 \}$$ for the unit ball of $k$.

Let $p$ be a prime number, $k$ be a $p$-adic local field. The $p$-adic valuation $v_p$ on $\Q$ extends naturally to $\Q_p$ and $k$, and we shall write $\mcO_k$ for the ring of integers of $k$. We shall call the  $p$-adic valuation of any generator of the different ideal of $\mcO_k$ over $\Z_p$ as the \textbf{different index} of $k$.

Let $e$ be the ramification index of $k$ over $\Q_p$, $\pi_k \in k$ be a uniformizer of $k$, $R_k$ be the group of multiplicative representatives in the unit group of $\mcO_k$. Then we have an group isomorphism 
$$k^\times \cong \pi_k^{\Z}\times R_k \times (1+\pi_k\mcO_k).$$
For each $x = \pi_k^a\cdot u\cdot (1+t) \in k^\times$ according to the this group isomorphism, 
the $p$-adic logarithm function $\log_p: k^\times \to k$ is defined by the formula
$$\log_p(x) = \log_p(1+t) = \sum_{n=1}^\infty (-1)^{n+1}\frac{t^n}{n}.$$
For each $t \in k^\times$ with $v_p(t)>\frac{1}{p-1}$, 
the $p$-adic exponent function $\exp_p:\{t\in k: v_p(t)>\frac{1}{p-1}\}\to k^\times$ is defined by the formula
$$\exp_p(t) = \sum_{n=0}^\infty \frac{t^n}{n!}. $$
For each $t\in k$ with $v_p(t)>\frac{1}{p-1}$, we have $\log_p(\exp_p(t)) = t.$

\hfill \break
\noindent\textbf{Curves:}

Let $K$ be a field of characteristic zero with an algebraic closure $\overline{K}$,
$E$ be an elliptic curve defined over $K$, $n\ge 1$ be an positive integer. 
Then we shall write $E[n]$ for the group of $n$-torsion points of $E_{\overline{K}}\defeq E\times_K \overline{K}$, and write $K(E[n])\subseteq \overline{K}$ for the ``$n$-torsion point field" of $E$, 
i.e. the field generated by $K$ and the coordinates of the points in $E[n]$.

\section{Modified Estimation for Log-volumes}

This section is devoted to the estimation for log-volumes in IUT theory.
We shall prove Proposition \ref{1-prop:Estimates the log-volume of initial Theta-data}, which is a modified version of \cite{IUTchIV}, Theorem 1.10 and \cite{ExpEst}, Theorem 5.1.
The majority of the results in this section can be viewed as slight modifications of \cite{IUTchIV}, Section 1 and \cite{ExpEst}, Theorem 5.1.

\subsection{Estimation for Local Log-volumes}
We shall begin with some estimation for local log-volumes.

\begin{lem} \label{1-lem:Estimates of Normalization}
Let $\{k_i\}_{i\in I}$ be a finite set of $p$-adic local fields [cf. \S0], where $|I|\ge 2$. 
For each $i\in I$, write 
$\pi_i$ for a uniformizer of $k_i$;
$\mcO_{k_i}$ for the ring of integers of $k_i$;
$e_i$ for the ramification index of $k_i$ over $\Q_p$;
$d_i \in \frac{1}{e_i}\Z$ for the different index [cf. \S0] of $k_i$.

(i) We shall write 
\begin{equation*}
R_I \defeq \bigotimes_{i\in I}\mcO_{k_i}, 
\; \log_p(R_I^\times) \defeq \bigotimes_{i\in I}\log_p(\mcO_{k_i}^\times),
\end{equation*}
where the tensor product is over $\Z_p$; write  $(R_I)^\sim$ for the normalization of $R_I$ in its ring of fractions $R'$.
Then we have 
\begin{equation*}
R' = \Q_p \otimes_{\Z_p} R_I = \Q_p \otimes_{\Z_p} \log_p(R_I^\times) = \bigotimes_{i\in I}k_i, 
\end{equation*}
where the last tensor product is over $\Q_p$.
We also have direct sum decompositions
$$R' = \bigoplus_{i=1}^{n_{R'}}L_i 
\quad \text{and} \quad (R_I)^\sim = \bigoplus_{i=1}^{n_{R'}}\mcO_{L_i},$$ 
where for each $1\le i\le n_{R'}$, $L_i$ is a $p$-adic local field with ring of integers $\mcO_{L_i}$.

(ii) For each
\begin{equation*}
\overrightarrow{\lambda} = (\lambda_i)_{i\in I} \in \Lambda_I \defeq \bigoplus_{i\in I}\frac{1}{e_i}\cdot\Z,
\end{equation*}
we shall call 
\begin{equation*}
\deg(\overrightarrow{\lambda}) \defeq \sum_{i\in I}\lambda_i 
\in \Lambda \defeq \frac{1}{\gcd_{i\in I}\{e_i\}}\Z
\end{equation*}
the degree of $\overrightarrow{\lambda}$, and write 
\begin{equation*}
p^{\overrightarrow{\lambda}} \defeq \otimes_{i\in I}\pi_i^{\lambda_ie_i} \in R_I.
\end{equation*}
Then the set $p^{\overrightarrow{\lambda}} \cdot (R_I)^\sim$ is determined by $\deg(\overrightarrow{\lambda})$ 
[i.e. independent of the choice of $\pi_i$, $\lambda_i$ for $i\in I$],
and we shall write 
\begin{equation*}
p^{\deg(\overrightarrow{\lambda})}\cdot (R_I)^\sim\defeq p^{\overrightarrow{\lambda}} \cdot (R_I)^\sim.
\end{equation*}
Moreover, for any $\lambda, \lambda'\in\Lambda$ with $\lambda \ge \lambda'$, 
we have $p^{\lambda}\cdot (R_I)^\sim \subseteq  p^{\lambda'}\cdot (R_I)^\sim$.

(iii) Suppose that $d_* = \max\limits_{i\in I}\{d_i\}$ for some $*\in I$. View $d_*\in \frac{1}{e_*}\cdot\Z \hookrightarrow \Lambda_I$ as an element $\overrightarrow{d_*}\in \Lambda_I$, and write 
$$\overrightarrow{d_I} \defeq (d_i)_{i\in I}\in \Lambda_I, 
\; \overrightarrow{d_I}' \defeq \overrightarrow{d_I} - \overrightarrow{d_*} \in \Lambda_I,$$
$$ d_I\defeq \deg(\overrightarrow{d_I}) = \sum_{i\in I}d_i,
\; d'_I\defeq \deg(\overrightarrow{d_I}') = d_I- \max_{i\in I}\{d_i\}.
$$
Then we have
$$p^{d'_I}(R_I)^\sim \subseteq R_I \subseteq (R_I)^\sim.$$
\end{lem}
\begin{proof}
Assertions (i) and (ii) can be proved by definition.
Assertion (iii) follows immediately from \cite{IUTchIV}, Proposition 1.1.
\end{proof}

\begin{lem} \label{1-lem: Estimates of Differents}
Let $k/k_0$ be a finite extension of $p$-adic local fields. Let $e$, $e_0$ be the ramification indices of $k$ and $k_0$ over $\Q_p$ respectively. Let $d$, $d_0$ be the different indices [cf. \S0] of $k$ and $k_0$ respectively. Then:

(i) We have $d + \frac{1}{e} \ge d_0 + \frac{1}{e_0} $.
Moreover, if $k$ is tamely ramified over $k_0$, then we have $d + \frac{1}{e} = d_0 + \frac{1}{e_0}$.

(ii) Suppose that $k$ is a finite Galois extension of an intermediate subfield $k_1$, such that $k_1$ is tamely ramiﬁed over $k_0$. Then we have $$d \le d_0 + \frac{1}{e_0} + v_p([k:k_1]). $$

(iii) For the case of $k_0=\Q_p$, we have:
\begin{itemize}
    \item If $p\nmid e$, then $d=1-\frac{1}{e}$; if $e=1$, then $d=0$.
    \item If $k$ is Galois over $\Q_p$, then $d\le 1+v_p(e)$.
\end{itemize}
\end{lem}
\begin{proof}
For assertions (i) and (ii), cf, \cite{IUTchIV}, Proposition 1.3.
For assertion (iii), we have $e_0 = 1$, $d_0 = 0$. 
Hence by (i), if $p\nmid e$, then $d=1-\frac{1}{e}$; if $e=1$, then $d=1-\frac{1}{e} = 0$.
Now suppose that $k$ is Galois over $\Q_p$.
Let $k_1$ be the maximal unramified subextension of $k_0$ in $k$, then $k$ is Galois over $k_1$, and we have $[k:k_1] = e$.
Hence by (ii), we have $d \le d_0 + \frac{1}{e_0} + v_p([k:k_1]) = 1 + v_p(e)$.
\end{proof}

\begin{lem} \label{1-lem:Estimates of Logarithms}
We continue to use the notation of Lemma \ref{1-lem:Estimates of Normalization}.
For each $i\in I$, define
\begin{equation*}
 a_i \defeq \frac{1}{e_i}\cdot\left\lceil \frac{e_i+1}{p-1} \right\rceil,
\; b_i \defeq \sup_{n \ge 0}\left\{ n - \frac{p^n}{e_i} \right\},
\end{equation*}
\begin{equation*}
\overrightarrow{a_I} \defeq (a_i)_{i\in I},  \overrightarrow{b_I} \defeq (b_i)_{i\in I}\in\Lambda_I,
\; a_I\defeq \deg(\overrightarrow{a_I}), \; b_I\defeq \deg(\overrightarrow{b_I}).
\end{equation*}

Let $$\phi:\log_p(R_I^\times) \xrightarrow{\sim} \log_p(R_I^\times)$$ 
be an automorphism of [finitely generated] $\Z_p$-module, which can be extended [by tensoring with $\Q_p$ over $\Z_p$] to an automorphism of 
$R' = \Q_p \otimes_{\Z_p} \log_p(R_I^\times).$
Then:

(i) For each $i\in I$, we have 
$$\pi_i^{a_i e_i}\mcO_{k_i} \subseteq \log_p(\mcO_{k_i}^\times) 
\subseteq \pi_i^{  - b_i e_i}\mcO_{k_i}. $$
For each $\overrightarrow{\lambda} \in \Lambda_I$ with degree $\lambda$, we have
$$p^{\lambda + a_I}\cdot R_I \subseteq p^{\overrightarrow{\lambda}}\cdot\log_p(R_I^\times)
\subseteq p^{\lambda - b_I}\cdot R_I .$$

(ii) For each $\overrightarrow{\lambda} \in \Lambda_I$ with degree $\lambda$, we have
$$\phi(p^\lambda \cdot (R_I)^\sim ) \subseteq p^{\lfloor \lambda-d'_I-a_I \rfloor}\log_p(R_I^\times)
\subseteq p^{\lfloor \lambda-d'_I-a_I \rfloor-b_I} \cdot (R_I)^\sim. $$
In particular, if $e_i \le p-2$ for each $i\in I$, then
$$\phi(p^\lambda \cdot (R_I)^\sim ) \subseteq p^{\lambda-d'_I-1} \cdot (R_I)^\sim. $$

(iii)  We shall call 
$\mathcal{I}_{k_i}\defeq \frac{1}{2p}\log_p(\mcO_{k_i}^\times)\subseteq k_i$ 
the log-shell of $k_i$ for $i\in I$, and call 
$$\mathcal{I}_{R'}\defeq \bigotimes_{i\in I}\mathcal{I}_{k_i} 
= p^{-|I|}\cdot 2^{-|I|}\cdot \log_p(R_I^\times) \subseteq R'$$
the log-shell of $R'$, where the tensor product is over $\Z_p$.
Then for each $\overrightarrow{\lambda} \in \Lambda_I$ with degree $\lambda$, we have
$$p^{\overrightarrow{\lambda}}\cdot \mathcal{I}_{R'}
\subseteq p^{\lambda-b_I-v_p(2p)\cdot |I|}(R_I)^\sim.$$
\end{lem}
\begin{proof}
First, to prove assertion (i), we only need to prove $$\pi_i^{a_i e_i}\mcO_{k_i} \subseteq \log_p(\mcO_{k_i}^\times) 
\subseteq \pi_i^{  - b_i e_i}\mcO_{k_i}$$ for each $i\in I$.
On one hand, for each $x\in \pi^{a_i e_i}\mcO_{k_i}$, since $a_i = \frac{1}{e_i}\cdot\left\lceil \frac{e_i+1}{p-1} \right\rceil 
> \frac{1}{e_i}\cdot \frac{e_i}{p-1} =\frac{1}{p-1}$, 
we have $v_p(x)>\frac{1}{p-1}$ and $x = \log_p(\exp_p(x)) \in \log_p(\mcO_{k_i}^\times)$. 
Hence $\pi_i^{ a_i e_i}\mcO_{k_i} \subseteq \log_p(\mcO_{k_i}^\times)$.
On the other hand, for each $x \in \pi_i\mcO_{k_i}$, we have
\begin{align*}
    v_p(\log_p(1+x)) & =v_p(\sum_{m=1}^\infty (-1)^{m+1}\frac{x^m}{m}) 
    \ge \inf_{m\ge1}\{v_p(\frac{x^m}{m})\} 
     \\ & = - \sup_{m\ge 1}\{v_p(m) - m\cdot v_p(x) \}
     \ge  - \sup_{n=v_p(m) \ge 0}\left\{ n - \frac{p^n}{e_i} \right\}
     = -b_i.
\end{align*}
Hence $\log_p(1+x) \in \pi_i^{-b_i e_i}\mcO_{k_i}$. Then by the definition of $p$-adic logarithm,
we have $\log_p(\mcO_{k_i}^\times) = \log_p(1+\pi_i\mcO_{k_i}) 
\subseteq \pi_i^{-b_i e_i}\mcO_{k_i}.$ 
This proves assertion (i).

Assertions (ii) and (iii) can be proved similarly to [the proof of] \cite{IUTchIV}, Proposition 1.2.
\end{proof}

\begin{lem} \label{1-lem:Estimation for p-adic log-volume}
We continue to use the notation of Lemma \ref{1-lem:Estimates of Logarithms}.

(i) Recall that $R'$ has a direct sum decomposition $R' = \bigoplus_{i=1}^{n_{R'}}L_i$, where $L_i$ $(1\le i\le n_{R'})$ are $p$-adic local fields. 
The log-volume functions on $L_i$ [cf. \cite{AbsTopIII}, Proposition 5.7 (i)] determines a normalized log-volume function $$\mu^{\log}(-)$$ defined on compact open subsets of $R'$, valued in $\R$, which is normalized so that 
$$\mu^{\log}((R_I)^\sim) = 0, \quad \mu^{\log}(p(R_I)^\sim) = -\log(p).$$

(ii) For a subset $A\subseteq R'$, the holomorphic hull of $A$ in $R'$ is the smallest subset of the form $\bigoplus_{i=1}^{n_{R'}}x_i\mcO_{L_i} \subseteq R'$ which contains $A$, 
where $x_i$ is a non-zero element of $L_i$ for each $1\le i\le n_{R'}$.

Suppose that $\overrightarrow{\lambda} \in \Lambda_I$ with degree $\lambda$. Let $U$ be holomorphic hull of the union of $p^{\overrightarrow{\lambda}}\cdot \mathcal{I}_{R'}$ and $\phi(p^{\overrightarrow{\lambda}}\cdot (R_I)^\sim )$,
where $\phi$ runs through all the automorphisms of the $\Q_p$-vector space $R'$ which is extended by an automorphism of the $\Z_p$-module $\log_p(R_I^\times)$ (i.e. all the possible $\phi$ in Lemma \ref{1-lem:Estimates of Logarithms}). 
Then we have
\begin{equation} \label{1.4-eq1}
    \mu^{\log}(U) \le (\max\{ \lceil -\lambda+d'_I+a_I\rceil, -\lambda+v_p(2p)\cdot |I| \} + b_I )\cdot \log(p).
\end{equation}
\end{lem}
\begin{proof}
For assertion (i), cf. \cite{IUTchIV}, Proposition 1.4. Next, we consider assertion (ii).
Write $$u \defeq \max\{ \lceil -\lambda+d'_I+a_I\rceil, -\lambda+v_p(2p)\cdot |I| \} + b_I. $$
By Lemma \ref{1-lem:Estimates of Logarithms} (ii), for each possible automorphism $\phi$, we have 
$$\phi(p^\lambda (R_I)^\sim ) \subseteq p^{-(\lceil -\lambda+d'_I+a_I\rceil+b_I)}(R_I)^\sim 
\subseteq p^{-u}\cdot (R_I)^\sim. $$
And by Lemma \ref{1-lem:Estimates of Logarithms} (iii), we have
$$p^{\overrightarrow{\lambda}}\cdot \mathcal{I}_{R'} \subseteq p^{-(-\lambda+b_I+v_p(2p)\cdot |I|)}(R_I)^\sim
\subseteq p^{-u}\cdot (R_I)^\sim. $$
Since $p^{-u}\cdot (R_I)^\sim$ is of the form $\bigoplus_{i=1}^{n_{R'}}x_i\mcO_{L_i}$ 
[with each $x_i\in L_i$ satisfying $v_p(x_i) = -u$], $U$ must be a subset of $p^{-u}\cdot (R_I)^\sim$ [by the definition of holomorphic hull]. Hence $$ \mu^{\log}(U) \le \mu^{\log}(p^{-u}\cdot (R_I)^\sim) = u\cdot \log(p). $$
This proves assertion (ii).
\end{proof}

We also have a archimedean version of Lemma \ref{1-lem:Estimation for p-adic log-volume}.

\begin{lem} \label{1-lem:Estimation for complex log-volume}
Let $\{k_i\}_{i\in I}$ be a finite set of complex archimedean local fields, where $|I|\ge 2$. For $i\in I$, write $\mcO_{k_i}$ for the ring of integers of $k_i$. Then:

(i) Write $R' \defeq \bigotimes_{i\in I}k_i$, where the tensor product is over $\R$.    
Then we have a naural direct sum decomposition 
\begin{align*}
    R' = \bigoplus_{i=1}^{n_{R'}}\C_i,
\end{align*} 
where $n_{R'}=2^{|I|-1}$, and $\C_i$ is a complex archimedean local field for each $1\le i\le n_{R'}$. 
Write $B_I \defeq  \bigoplus_{i=1}^{n_{R'}}\mcO_{\C_i},$ where $\mcO_{\C_i}$ is the unit ball of $\C_i$.

(ii) The radial log-volume functions on $\C_i$ [cf. \cite{AbsTopIII}, Proposition 5.7 (ii)] determines a normalized log-volume function $$\mu^{\log}(-)$$ defined on compact subsets of $R'$, valued in $\R$, which is normalized so that 
$$\mu^{\log}(B_I) = 0,\quad \mu^{\log}(e\cdot B_I) = 1.$$

(iii) We shall call $I_{k_i}\defeq \pi\cdot\mcO_{k_i}\subseteq k_i$ the log-shell of $k_i$ for $i\in I$. 
We shall write $I_{R'}$ for the image of $\prod_{i\in I}I_{k_i}$ in $R'$
under the natural homomorphism $\prod_{i\in I}k_i \to \bigotimes_{i\in I}k_i = R'$, and call $I_{R'}$ the log-shell of $R'$.

For a subset $A\subseteq R'$, the holomorphic hull of $A$ in $R'$ is the smallest subset of the form $\bigoplus_{i=1}^{n_{R'}}x_i\mcO_{\C_i} \subseteq R'$ which contains $A$, 
where $x_i$ is a non-zero element of $\C_i$ for each $1\le i\le n_{R'}$.

Let $U$ be holomorphic hull of the union of $\mathcal{I}_{R'}$ and $B_I$,
Then we have $$\mu^{\log}(U) \le |I|\cdot \log(\pi).$$
\end{lem}
\begin{proof}
For assertion (i), cf. \cite{IUTchIV}, Proposition 1.5, where we can take $|V|=1$.
Assertion (ii) can be proved by definition. 
Assertion (iii) follows from the fact that $I_{R'}, B_I \subseteq \pi^{|I|}\cdot B_I$ and $\mu^{\log}(\pi^{|I|}\cdot B_I) = |I|\cdot \log(\pi)$, which can be verified by means of routine and elementary arguments.
\end{proof}

\subsection{Estimation for Log-volumes of Initial \texorpdfstring{$\Theta$}{}-Data }
In this subsection, we will estimate the log-volumes associated to $\mu_6$-initial $\Theta$-data [cf. \cite{IUTchI}, Definition 3.1 and \cite{ExpEst}, Definition 4.1]. 
We will maintain the notation throughout this subsection, and introduce additional notation as necessary to support our analysis.

The following lemma is helpful for constructing $\mu_6$-initial $\Theta$-data.
\begin{lem} \label{1-lem: initial Theta-data}
Suppose that the following conditions are satisfied:
\begin{itemize}
    \item[(a)] $F$ is a number field such that $\sqrt{-1}\in F$; $\overline{F}$ is an algebraic closure of $F$. 
    Write $G_F \defeq \Gal(\overline{F}/F)$ for the absolute Galois group of $F$.
    \item[(b)] $X_F$ is a punctured elliptic curve over $F$. Write $E_F$ for the elliptic curve over $F$ determined by $X_F$, then $E_F$ is semi-stable and $F = F(E_F[6])$.
    \item[(c)] $F_{\tmod} \subseteq F$ is the field of moduli [cf., e.g., \cite{AbsTopIII}, Definition 5.1, (ii)] of $X_F$. 
    Write $$\V_{\tmod} \defeq \V(F_{\tmod}), \quad \V_{\tmod}^{\non} \defeq \V(F_{\tmod})^{\non}. $$
    Then $$\V_{\tmod}^{\bad} \subseteq \V_{\tmod}^{\non} $$ is a nonempty set of nonarchimedean valuations of $F_{\tmod}$, such that $E_F$ has bad multiplicative reduction at each $v\in \V(F)$ lying over $\V_{\tmod}^{\bad}$. 
    Write $$\V(F)^{\bad}\defeq \V_{\tmod}^{\bad}\times_{\V_{\tmod}} \V(F).$$
    \item[(d)] $l\ge 5$ is a prime number, such that the image of the Galois representation $$\rho_{E_F,l}: G_F\to \GL(2,\F_l)$$ determined by the $l$-torsion points of $E_F$ contains the subgroup $\SL(2,\F_l) \subseteq \GL(2,\F_l)$. 
    Write $K\defeq F(E_F[l])\subseteq \overline{F}$ for the $l$-torsion point field of $E_F$, then $K$ is a finite Galois extension of $F$.
     Write $X_K \defeq X_F\times_F K$, then $X_K$ admits a $K$-core.
    \item[(e)] The field extension $F / F_{\tmod}$ is Galois of degree prime to $l$. 
    Also, for each $v\in \V(F)^{\bad}$, $l$ is prime to the residue characteristic of $v$, as well as to the order of the $q$-parameter of $E_F$ at $v$.
\end{itemize}
Then there exists a section $$\eta : \V_{\tmod} \xrightarrow{\sim} \uV \subseteq \V(K)$$ of the natural surjection $\V(K) \twoheadrightarrow \V_{\tmod}$,
a hyperbolic orbicurve $\underline{C}_K$ of type $(1,l\text{-tors})_{\pm}$ [cf. \cite{ExpEst}, Definition 3.4, (i)] over $K$ and a cusp  $\underline{\epsilon}$ of $\underline{C}_K$,
such that $$(\overline{F}/F,X_F,l,\underline{C}_K, \uV,\V_{\tmod}^{\bad}, \underline{\epsilon})$$ is a $\mu_6$-initial $\Theta$-data.
\end{lem}
\begin{proof}
The lemma is a consequence of \cite{IUTchI}, Definition 3.1 and \cite{ExpEst}, Definition 4.1, where the existence of the section $\eta : \V_{\tmod} \xrightarrow{\sim} \uV$ follows from the condition that the image of $\rho_{E_F,l}$ contains the subgroup $\SL(2,\F_l) \subseteq \GL(2,\F_l)$. 
\end{proof}

\begin{definition} \label{1-def:Log of arithmetic divisors}
Suppose that $(\overline{F}/F,X_F,l,\underline{C}_K, \uV,\V_{\tmod}^{\bad}, \underline{\epsilon})$ is a $\mu_6$-initial $\Theta$-data.
We shall proceed with the notation of Lemma \ref{1-lem: initial Theta-data},

(i) For each $v\in\V_{\tmod}^{\non}$, we shall write $\uv\defeq \eta(v)\in \V(K)$,
and write $p_v$ for the residue characteristic of  $v$,
$e_v$ for the ramification index of $K_{\uv}$ over $\Q_p$,
$d_v \in \frac{1}{e_v}\Z$ for the different index [cf. \S0] of $K_{\uv}$.  

When $v\in \V_{\tmod}^{\bad}$, write $q_v$ for the $q$-parameter of the Tate curve $E_{K_{\uv}} \defeq E_F \times_F K_{\uv}$, then by the fact that $2 l$-torsion points of $E_F$ are defined over $K$ and Proposition \ref{2-prop: torsion point field} (iv),
we have
\begin{equation} \label{1.8-eq: 2l divides e times v_p}
2l \mid \uv(q_v) = e_v\cdot v_{p_v}(q_v).
\end{equation}
When $v\notin \V_{\tmod}^{\bad}$, write $q_v\defeq 1 $, then $ v_{p_v}(q_v) = 0$. Hence (\ref{1.8-eq: 2l divides e times v_p}) is valid for every $v\in\V_{\tmod}^{\non}$.

We shall also write $$ a_v \defeq \frac{1}{e_v}\cdot\left\lceil \frac{e_v+1}{p_v-1} \right\rceil,
\quad b_v \defeq \sup_{n \ge 0}\left\{ n - \frac{p_v^n}{e_v} \right\} .$$
Then we have $a_v = -b_v = \frac{1}{e_v}$ when $e_v\le p_v-2$.

(ii) For each $v_p\in \V_{\Q}^{\non}$, we shall put
{\small
\begin{equation*} 
\log(\mathfrak{q}_{v_p}) 
\defeq \frac{1}{[F_{\tmod}:\Q]}\sum\limits_{v\in (\V_{\tmod})_{v_p}}\deg(v(q_v)\cdot v) \in \R.
\end{equation*}
}

We shall also write 
{\small\begin{equation*} 
\mathfrak{q} \defeq \sum_{v\in \V_{\tmod}^{\non}} v(q_v) \cdot v \in \ADiv_{\R}(\V_{\tmod}), 
\quad \log(\mathfrak{q})\defeq \udeg_{F_{\tmod}}(\mathfrak{q})  \in \R.
\end{equation*}
}
Then 
\begin{equation} \label{1.8-eq: log(q)}
\log(\mathfrak{q})
= \frac{1}{[F_{\tmod}:\Q]}\sum\limits_{v\in \V_{\tmod}^{\non}}\deg(v(q_v)\cdot v)
=  \sum_{v_p\in \V_{\Q}^{\non}} \log(\mathfrak{q}_{v_p}) .
\end{equation}

\end{definition}

\begin{definition} \label{1-def:Log-volume of initial Theta-data}
We shall proceed with the above notation.
Suppose that $j\in\{1,2,\dots,l^\divideontimes\!=\!\frac{l-1}{2}\}$, $v_{\Q}\in \V_{\Q}$.

(i) Let $v_0,v_1,\dots,v_j\in (\V_{\tmod})_{v_{\Q}}$ be $j+1$ valuations. For $0\le i\le j$, write $\uv_i \defeq \eta(v_j)\in \uV$ and $k_i \defeq K_{\uv_i}$.
We shall use $-|\log(\Theta)|_{(v_0,v_1,\dots,v_j)}$ to denote the log-volume of the following:
\begin{itemize}
    \item For $v_{\Q} \in \V_{\Q}^{\non}$, the holomorphic hull described in Lemma \ref{1-lem:Estimation for p-adic log-volume}, 
    where by (\ref{1.8-eq: 2l divides e times v_p}) we can write $\overrightarrow{\lambda} \defeq \frac{j^2}{2l}\cdot \uv_j(q_{v_j}) \in \frac{1}{e_{v_j}} \Z \hookrightarrow \oplus_{i=0}^j \frac{1}{e_{v_i}} \Z .$

    \item For $v_{\Q} \in \V_{\Q}^{\arc}$, the holomorphic hull described in Lemma \ref{1-lem:Estimation for complex log-volume}. In this case, since $\sqrt{-1}\in F\subseteq K$, all the $k_i$ are complex archimedean local fields, which satifies the premise of Lemma \ref{1-lem:Estimation for complex log-volume}.
\end{itemize}

(ii) The weighted average log-volume $-|\log(\uu\Theta)|_{v_{\Q},j}$ is defined by the formula
\begin{equation} \label{1.9-eq1}
-|\log(\uu\Theta)|_{v_{\Q},j} \defeq 
\frac{ \sum\limits_{v_0,\dots,v_j\in (\V_{\tmod})_{v_{\Q}}} \big(\prod\limits_{0\le i\le j}[(F_{\tmod})_{v_i}:\Q_{v_{\Q}}]\big)\cdot \big(-|\log(\Theta)|_{(v_0,v_1,\dots,v_j)}\big)}
{\sum\limits_{v_0,\dots,v_j\in (\V_{\tmod})_{v_{\Q}}} \big(\prod\limits_{0\le i\le j}[(F_{\tmod})_{v_i}:\Q_{v_{\Q}}]\big)}, 
\end{equation}
and the procession-normalized log-volume $-|\log(\uu\Theta)|_{v_{\Q}}$ is defined by the formula
\begin{equation} \label{1.9-eq2}
-|\log(\uu\Theta)|_{v_{\Q}} \defeq \frac{1}{l^\divideontimes}\sum_{1\le j\le l^\divideontimes} (-|\log(\uu\Theta)|_{v_{\Q},j}).
\end{equation}

(iii) The log-volume of $\mu_6$-initial $\Theta$-data $-|\log(\uu\Theta)|$ is defined by the formula
\begin{equation} \label{1.9-eq3}
-|\log(\uu\Theta)| \defeq - \sum_{v_{\Q}\in\V_{\Q}} |\log(\uu\Theta)|_{v_{\Q}}.
\end{equation}

\end{definition}

\begin{prop} \label{1-prop: IUTT main result}
Proceeding with the above notation, we have  
\begin{equation} \label{1.10-eq1}
-\frac{1}{
    2l}\cdot\log(\mathfrak{q}) \le -|\log(\uu\Theta)|. 
\end{equation}
\end{prop}
\begin{proof}
The proposition follows from the $\mu_6$-version of \cite{IUTchIII}, Corollary 3.12.
\end{proof}

From now on, we shall focus on the special case of $F_{\tmod} = \Q$, which is enough for the present paper. Then the $j$-invarient $j(E_F)$ of $E_F$ is a rational number.
For each $v_p \in \V_{\Q}^{\non}$, 
write $ e_p \defeq e_{v_p},\; a_p \defeq a_{v_p},\;  b_p \defeq b_{v_p},\; d_p \defeq d_{v_p},\;  q_p \defeq q_{v_p}. $ Then $\log(\mathfrak{q}_{v_p}) = v_p(q_p) \cdot \log(p)$ for any prime number $p$, and by (\ref{1.8-eq: log(q)}) we have 
\begin{equation} \label{1.10-eq3: log(q)}
\log(\mathfrak{q}) = \sum_{p}  v_p(q_p) \cdot \log(p).
\end{equation}

\begin{prop} \label{1-prop:Estimates the log-volume of initial Theta-data}
Proceeding with the above notation, we have  
{\small
\begin{equation}  \label{1.11-eq1}
\begin{aligned}
    \frac{1}{6}\log(\mathfrak{q})
    \le & \frac{l^2+5l}{l^2+l-12}\cdot \big(\log(\pi) + \sum_{p\ge 2} d_p \cdot \log(p) + \sum_{e_p \ge p-1 } ( \frac{1}{p-1}+1-\frac{p-1}{e_p}) \cdot \log(p) 
     \\ &+ \sum_{e_p > p(p-1)} \log(\frac{e_p} {p-1} ) \big).
\end{aligned}
\end{equation} 
}
\end{prop}
\begin{proof}
By (\ref{1.9-eq3}) and (\ref{1.10-eq1}), we have 
{\small
\begin{equation}  \label{1.11-eq2}
 -\frac{1}{2l} \cdot \log(\mathfrak{q})\le -|\log(\uu\Theta)| 
 = (-|\log(\uu\Theta)|_{v_{\R}}) +  \sum_{v_p\in \V_{\Q}^{\non}} (-|\log(\uu\Theta)|_{v_p} ),
\end{equation} 
}
where $v_{\R}$ is the unique real archimedean valuation on $\Q$. 

We shall estimate $-|\log(\uu\Theta)|_{v_p}$ at first.
Let $v_p \in \V_{\Q}^{\non}$ and $1 \le j \le l^\divideontimes$.
Since $F_{\tmod} = \Q$, we have $(\V_{\tmod})_{v_p} = \{ v_p\}$. 
Then by (\ref{1.4-eq1}) and (\ref{1.9-eq1}), we have
{\small
\begin{equation} \label{1.11-eq3}
\begin{aligned}
   -|\log(\uu\Theta)|_{v_p, j} = & -|\log(\uu\Theta)|_{(\underbrace{v_p, \dots, v_p}_{j+1})}
 \le  \big(\max \{ \lceil -\frac{j^2}{2l}\cdot v_p(q_p) + j\cdot d_p + (j+1)\cdot a_p \rceil, 
   \\  & -\frac{j^2}{2l} \cdot v_p(q_p) + v_p(2p)\cdot(j+1) \big\} + (j+1) b_p \big)\cdot \log(p).
\end{aligned}
\end{equation}
}

By (\ref{1.8-eq: 2l divides e times v_p}) we have $2l \mid e_p \cdot v_p(q_p)$, 
hence $-\frac{j^2}{2l}\cdot v_p(q_p)\in \frac{1}{e_p} \Z$ and $-\frac{j^2}{2l}\cdot v_p(q_p) + j\cdot d_p + (j+1)\cdot a_p \in \frac{1}{e_p}\Z$.
Then since $d_p \ge 1 -  \frac{1}{e_p}$ [cf. Lemma \ref{1-lem: Estimates of Differents}], 
we have
\begin{equation} \label{1.11-eq4}
	\lceil -\frac{j^2}{2l}\cdot v_p(q_p) + j\cdot d_p + (j+1)\cdot a_p \rceil \le -\frac{j^2}{2l}\cdot v_p(q_p) +(j+1)(d_p+a_p) .
\end{equation}
Meanwhile, since $d_p \ge 1 -  \frac{1}{e_p}$ and $a_p = \frac{1}{e_p}\cdot\left\lceil \frac{e_p+1}{p-1} \right\rceil \ge v_p(2) + \frac{1}{e_p}$, we have $d_p + a_p \ge v_p(2p)$, hence 
\begin{equation} \label{1.11-eq5}
	 -\frac{j^2}{2l} \cdot v_p(q_p) + v_p(2p)\cdot(j+1) \le -\frac{j^2}{2l}\cdot v_p(q_p) +(j+1)(d_p+a_p) .
\end{equation}

By (\ref{1.11-eq3}),  (\ref{1.11-eq4}) and  (\ref{1.11-eq5}), we have 
\begin{equation*}
-|\log(\uu\Theta)|_{v_p, j}
    \le   \left(-\frac{j^2}{2l}\cdot v_p(q_p) + (j+1) (d_p + a_p + b_p) \right) \cdot \log(p)  .
\end{equation*}
Then by averaging the cases $1\le j \le l^\divideontimes$ and by (\ref{1.9-eq2}), we have
\begin{equation} \label{1.11-eq6}
 -|\log(\uu\Theta)|_{v_p} \le \big( - \frac{l+1}{24} \cdot v_p(q_p) + \frac{l+5}{4} \cdot (d_p + a_p + b_p) \big) \cdot \log(p).
\end{equation}

Similarly, by combining Lemma \ref{1-lem:Estimation for complex log-volume} (iii) with Definition \ref{1-def:Log-volume of initial Theta-data} (i), (ii), we can get
\begin{equation*}
-|\log(\uu\Theta)|_{v_{\R}} \le \frac{l+5}{4} \cdot \log(\pi).
\end{equation*}
Then by (\ref{1.11-eq2}) and  (\ref{1.11-eq6}), we have
{\small
\begin{equation*}
\begin{aligned}
    &-\frac{1}{2l} \cdot \log(\mathfrak{q}) \le -|\log(\uu\Theta)| 
    =  \sum_{v_p\in\V_{\Q}^{\non}} (-|\log(\uu\Theta)|_{v_p} ) + (-|\log(\uu\Theta)|_{v_{\R}})
    \\ \le& -\frac{l+1}{24}\cdot\sum_{p} v_p(q_p)\cdot \log(p) 
    +\frac{l+5}{4} \cdot\big(\log(\pi)+\sum_{p} (d_p + a_p + b_p)\cdot\log(p)  \big) .
\end{aligned}
\end{equation*}
}
Hence by (\ref{1.10-eq3: log(q)}) we have
\begin{equation} \label{1.11-eq7}
    \frac{1}{6}\log(\mathfrak{q})
    \le \frac{l^2+5l}{l^2+l-12}\cdot \big(\log(\pi) + \sum_{p\ge 2}(d_p + a_p + b_p)\cdot\log(p) \big) .
\end{equation}

To prove (\ref{1.11-eq1}), we are left to estimate $a_p + b_p$. 
Recall that
\begin{align*}
	a_p = \frac{1}{e_p}\cdot\left\lceil \frac{e_p+1}{p-1} \right\rceil \le \frac{1}{p-1} + \frac{1}{e_p},
	\quad b_p = \sup_{n \ge 0}\left\{ n - \frac{p^n}{e_p} \right\}.
\end{align*}
Since $n-\frac{p^n}{e_p} \ge (n+1) - \frac{p^{n+1}}{e_p} \Leftrightarrow p^n \ge \frac{e_p}{p-1} \Leftrightarrow n \ge \log_p(\frac{e_p} {p-1} )$,
when $n_0 \defeq \lceil \log_p(\frac{e_p} {p-1} ) \rceil \ge 1$, we have
\begin{align*}
	b_p = n_0 - \frac{p^{n_0} }{e_p} \le \log_p(\frac{e_p} {p-1}) + 1 - \frac{p}{e_p}
	\le \frac{\log(\frac{e_p}{p-1}  )}{\log(p)} + 1 - \frac{p}{e_p} .
\end{align*}

When $e_p \le p-2$, we have $a_p = - b_p = \frac{1}{e_p}$, hence $a_p + b_p = 0$;
when $p-1 \le e_p \le p(p-1)$, we have $b_p = 1 - \frac{p}{e_p}$, hence $a_p + b_p \le \frac{1}{p-1}+1-\frac{p-1}{e_p}$; 
when $e_p > p(p-1)$, we have  $a_p + b_p \le \frac{1}{p-1}+1-\frac{p-1}{e_p}+\frac{\log(e_p / (p-1) ) }{\log(p) }$.
Hence by (\ref{1.11-eq7}) we have
{\small
\begin{equation*}
\begin{aligned}
    \frac{1}{6}\log(\mathfrak{q})
    \le &\frac{l^2+5l}{l^2+l-12}\cdot \big( \log(\pi) + \sum_{p\ge 2} d_p \cdot \log(p) + \sum_{e_p \ge p-1 } ( \frac{1}{p-1}+1-\frac{p-1}{e_p}) \cdot \log(p)
    \\ &+ \sum_{e_p > p(p-1)} \log(\frac{e_p} {p-1} ) \big) .
\end{aligned}
\end{equation*}
}
\end{proof}

\subsection{Ramification datasets}
We shall introduce the concept of ramification dataset. 
Ramification dataset encodes the ramification information of an initial $\Theta$-data,
which is useful for explicit computation.

\begin{definition} \label{1-def: ramification dataset}
A \textbf{ramification dataset} $\mathfrak{R}$ consists of the following data:
\begin{itemize}
\item A prime number $l_0\ge 5$ called base prime, an positive integer $e_0$ called base index.
\item A non-empty finite set $S_{\gen}^{\multi}$ of positive integers called the set of general multiplicative indices, such that every element in $S_{\gen}^{\multi}$ equals $e$ or $e\cdot l$, where $1\le e\le e_0$ and $l\nmid e$.
\item A finite set of prime numbers $S_0$ called the set of special primes.
\item For each $p$ in $S_0$, a finite set $S_p^{\good}$ of positive integers called the set of good indices at $p$, and a finite set $S_p^{\multi}$ of positive integers called the set of multiplicative indices at $p$.
\end{itemize}
\end{definition}

\begin{definition} \label{1-def: ramification dataset for initial theta-data}
Let $\mathfrak{D} = (\overline{F}/F,X_F,l,\underline{C}_K, \uV,\V_{\tmod}^{\bad}, \underline{\epsilon})$ be a $\mu_6$-initial $\Theta$-data such that $F_{\tmod} = \Q$, 
with ramification indices $e_p$ defined as in the notation of Proposition \ref{1-prop:Estimates the log-volume of initial Theta-data}. 
Let $\eta(v_p)$ be the valuation in $\uV$ lying over $v_p$, cf. Lemma \ref{1-lem: initial Theta-data}.

(i)  Write $N$ for the denominator of the $j$-invarient $j(E_F) \in F_{\tmod} = \Q$ of the elliptic curve $E_F$ associated to $X_F$.
Then we shall say $\mathfrak{D}$ is \textbf{of type} $\bm{(l,N,N')}$ if we have $\log(N') = \log(\mathfrak q)$ for some positive integer $N'$.

(ii) We shall say $\mathfrak{D}$ \textbf{admits} a ramification dataset $\mathfrak{R}$ if in the notation of Definition \ref{1-def: ramification dataset}, the following conditions hold:
\begin{itemize}
\item $l_0 = l$, i.e. $l$ is the base prime of $\mathfrak{R}$.
\item For each $p \notin S_0$,  if $E_K$ has good reduction at $\eta(v_p)$ [i.e. $p\nmid N$], then $e_p = 1$; if $E_K$ has multiplicative reduction at $\eta(v_p)$ [i.e. $p\mid N$], then $e_p \in S_{\gen}^{\multi}$.
\item For each $p \in S_0$, if $E_K$ has good reduction at $\eta(v_p)$ [i.e. $p\nmid N$], then $e_p \in S_p^{\good}$; if $E_K$ has multiplicative reduction at $\eta(v_p)$ [i.e. $p\mid N$], then $e_p \in S_p^{\multi}$.
\end{itemize}
\end{definition}

\begin{cor} \label{1-cor: The log-volume of ramification dataset}
Let $\mathfrak{R}$ be a ramification dataset with base prime $l\ge 5$ and base index $e_0$. 
Then there exists an algorithm [whose construction is presented in the proof] to compute a real number $\Vol(\mathfrak{R}) \ge 0$ that depends only on $\mathfrak{R}$, such that for any $\mu_6$-initial $\Theta$-data  $\mathfrak{D}$ which admits $\mathfrak{R}$ and is of type $(l, N, N')$, proceeding with the notation of Proposition \ref{1-prop:Estimates the log-volume of initial Theta-data}, we have
{\small
\begin{equation} \label{1.14-eq: algorithm of log-volume}
    \frac{1}{6}\log(N') \le  \frac{l^2+5l}{l^2+l-12}\cdot \big( (1-\frac{1}{e_0  l})\cdot \log\rad(N) - \frac{1}{e_0}(1- \frac{1}{l}) \log\rad(N_l)  \big) + \Vol(\mathfrak{R}),
\end{equation}
}
where $N_l \defeq \prod_{p: l\mid v_p(N)} p^{v_p(N)}$.
\end{cor}
\begin{proof}
We continue to use the notation of Proposition \ref{1-prop:Estimates the log-volume of initial Theta-data} and Definition \ref{1-def: ramification dataset}. We shall provide an algorithm for computing the value of $\Vol(\mathfrak{R})$:

\textbf{Step 0:} 
Let  $p \le e_0 l +1$ be a prime number. 
When $p\notin S_0$, take $e = 1, \delta = 0$ or $e\in S_{\gen}^{\multi}, \delta = 1$; 
when $p\in S_0$, take $e\in S_p^{\good}, \delta = 0$ or take $e\in S_p^{\multi}, \delta = 1$. Then we get finitely many triples $(p, e, \delta)$.

\textbf{Step 1}: 
For each triple $(p, e, \delta)$ in step 0, we shall define several functions.
Write $a_p(e) \defeq \frac{1}{e}\cdot\left\lceil \frac{e+1}{p-1} \right\rceil$, $b_p(e) \defeq \sup_{n \ge 0}\left\{ n - \frac{p^n}{e} \right\}$. When $p\nmid e$, write $d_p(e) \defeq 1-\frac{1}{e}$; when $p\mid e$, write $d_p(e) \defeq 1+v_p(e)$.
When $\delta = 0$, take $u = 0$; when $\delta = 1$, let $u$ be any integer such that $0\le u \le 2l-1$ and $2l \mid e\cdot u$. 
When $l \mid e$, write $l_0(e) = l$; when  $l \nmid e$, write $l_0(e) = 1$. 
Then for $e\in S_{\gen}^{\multi}$, we have $e \le e_0 \cdot l_0(e)$.
For each integer $1\le j \le l^\divideontimes$, define
{\small
\begin{equation*}
\begin{aligned}
B_0(p, e, \delta, u, j) \defeq & \max\{ \lceil -\frac{j^2}{2l} \cdot u + j\cdot d_p(e) + (j+1)\cdot a_p(e) \rceil  + \frac{j^2}{2l}\cdot u, v_p(2p)\cdot(j+1) \}  \\& + (j+1)\cdot b_p(e),
\\
B_1(p, e, \delta, u) \defeq 
& \frac{1}{l^\divideontimes} \cdot \big(\sum_{j=1}^{l^\divideontimes} B_0(p, e, \delta, u, j) \big) \cdot \frac{4}{l+5} - (1-\frac{1}{e_0\cdot l_0(e)}) \cdot \delta,
\\ B_2(p) \defeq & \max_{(e, \delta, u)}\{ B_1(p, e, \delta, u) \}.
\end{aligned}
\end{equation*}
}

\textbf{Step 2}: 
Finally, define $\Vol(\mathfrak{R})$ by the equation
\begin{equation} \label{1-3-eq: definition of Vol(R)}
\Vol(\mathfrak{R}) \defeq \max\{0,\frac{l^2 + 5l}{l^2 + l - 12}\cdot \big( \log(\pi) + \sum_{p\le e_0 l +1\,\text{or}\, p\in S } B_2(p) \cdot \log(p) \big) \}.
\end{equation}

We claim that (\ref{1.14-eq: algorithm of log-volume}) is valid for the $\Vol(\mathfrak{R})$ defined in this way. To show this, for each prime number $p$, we shall write $\delta_p = 0$ if $p\nmid N$, and write $\delta_p = 1$ if $p\mid N$. 
Recall that since $l\nmid [F:\Q]$, we can see that $l\mid e_p$ if and only if $E_K$ has multiplicative reduction at $\eta(v)$ and $l \nmid v_p(N)$, which is equivalent to $(\frac{1}{l_0(e)} - \frac{1}{l}) \cdot \delta_p = (1-\frac{1}{l}) \delta_p \neq 0$.
Hence
\begin{equation} \label{1.14-eq3}
\log\rad(N) = \sum_{p} \delta_p \log(p), 
\; (1-\frac{1}{l}) \log\rad(N_l) = \sum_p (\frac{1}{l_0(e)} - \frac{1}{l} ) \cdot \delta_p \log(p).
\end{equation}

Write $u_p$ for the remainder of $v_p(q_p)$ modulo $p$.
Then we have $a_p = a_p(e_p)$, $b_p = b_p(e_p)$, $d_p \le d_p(e_p)$; 
$\delta_p = 0$ (resp. $\delta_p = 1$) if and only if $E_k$ has good (resp. multiplicative) reduction at $\eta(v_p)$;
$u_p = 0$ when $\delta_p = 0$, and $2l \mid e_p \cdot u_p$ when $\delta_p = 1$.
Hence $(p, e_p, \delta_p, u_p)$ is in the domain of the function $B_1(p, e, \delta, u)$. 

Similar to the proof of Proposition \ref{1-prop:Estimates the log-volume of initial Theta-data}, by (\ref{1.4-eq1}) and (\ref{1.9-eq1}), for $1\le j \le l^\divideontimes$, we have
\begin{equation*}
 -|\log(\uu\Theta)|_{v_p, j} + \frac{j^2}{2l} \cdot v_p(q_p)
 \le B_0(p, e_p, \delta_p, u_p, j) \cdot \log(p).
\end{equation*}
Then by averaging the cases $1\le j \le l^\divideontimes$, we have
{\small
\begin{equation*}
\begin{aligned}
-|\log(\uu\Theta)|_{v_p} + \frac{l+1}{24} \cdot v_p(q_p)
& \le  \frac{l+5} {4}\cdot \big( B_1(p, e_p, \delta_p, u_p)  +  (1-\frac{1}{e_0 l_0(e)}) \delta_p  \big) \cdot \log(p)
\\ & \le \frac{l+5}{4}\cdot \big( B_2(p)  +  (1-\frac{1}{e_0 l}) \delta_p - \frac{1}{e_0}(\frac{1}{l_0(e)} - \frac{1}{l} )\delta_p  \big) \cdot \log(p).
\end{aligned}
\end{equation*}
}
Hence by summing over $p$ and $v_{\R}$, and by (\ref{1.14-eq3}) [cf. the proof of Proposition \ref{1-prop:Estimates the log-volume of initial Theta-data}], we can get
{\small
\begin{equation} \label{1.14-eq2}
\begin{aligned}
    \frac{1}{6}\log(\mathfrak{q}) \le \frac{l^2+5l}{l^2+l-12}\cdot &\big( (1-\frac{1}{e_0l})\cdot \log\rad(N) - \frac{1}{e_0}(1- \frac{1}{l}) \log\rad(N_l)  
    \\& + \log(\pi) + \sum_{p} B_2(p) \cdot \log(p) \big).
\end{aligned}
\end{equation}
}

By (\ref{1.14-eq2}) and the definition of $\Vol(\mathfrak R)$, to prove (\ref{1.14-eq: algorithm of log-volume}), it suffices to show that $B(p) = 0$ when $p \notin S_0$ and $p \ge e_0 l + 2 \ge 3$.
For such $p$, consider all the possible $B_1(p,e,\delta,u)$. 

Recall that since $p\notin S$, we have $e = 1, \delta = 0$ or $e\in S_{\gen}^{\multi}, \delta = 1$.

When $e=1, \delta=0$, sicne $p \ge 3$, we have $a_p(e) = 1, b_p(e) = -1, d_p(e) = 0, u = 0, v_p(2p) = 1$. Hence $B_0(p,e,\delta,u,j) = 0$ and $B_1(p,e,\delta,u) = 0$.

When $e\in S_{\gen}^{\multi}, \delta = 1$, since $p \ge e_0 l + 2 \ge e + 2$, we have $a_p(e) = \frac{1}{e}, b_p(e) = -\frac{1}{e}, d_p(e) = 1-\frac{1}{e}, v_p(2p) = 1$. Then similar to the proof of (\ref{1.11-eq6}), we can show that
\begin{equation*} 
B_0(p,e,\delta,u,j) \le \frac{l+5}{4} \cdot (1-\frac{1}{e}), \; B_1(p,e,\delta,u) \le (1-\frac{1}{e}) - (1 - \frac{1}{e_0 l_0(e)}) \le 0.
\end{equation*}

Hence when $p \notin S$ and $p \ge e_0 l + 2$, we have 
\begin{equation*}
B_2(p) = \max_{(e, \delta, u)}\{ B_1(p, e, \delta, u) \} = 0.
\end{equation*} 
This proves the corollary.
\end{proof}

\begin{tiny-remark} \label{1-rmk: The log-volume of ramification dataset}
(i) Let $\mathfrak{R}$ be the the ramification dataset in \ref{1-cor: The log-volume of ramification dataset}, and let $p$ be a prime number.
By taking $n=0$, we have $b_p(e) \ge n - \frac{p^n}{e} = -\frac{1}{e}$, hence
in the definition of $B_0(p,e,\delta,u,j)$, we have $B_0(p,e,\delta,u,j) \ge (j+1)(v_p(2p)+b_p(e)) \ge 0$. Then we have $B_1(p,e,\delta,u) \ge  (\frac{1}{e_0 l} - 1) \delta$ and $B_2(p) \ge \max\limits_{(e, \delta, u)} \{(\frac{1}{e_0 l} - 1) \delta\} \ge -1$.

Suppose that $p \notin S_0$, or  $S_p^{\good}$ is not empty, then there exists a tuple $(p,e,\delta,u)$ in the domain of $B_1(p,e,\delta,u)$ with $\delta = 0$. 
Hence we have $B_2(p) \ge \max\limits_{(e, \delta, u)} \{(\frac{1}{e_0 l} - 1) \delta\} = 0$.

(ii) Suppose that $S_0$ is empty, or $S_p^{\good}$ is not empty for each $p\in S_0$. Then for each prime number $p$,we have $B_2(p) \ge 0$. Hence by (\ref{1-3-eq: definition of Vol(R)}) we have 
$\Vol(\mathfrak{R}) \ge \max\{0,\frac{l^2 + 5l}{l^2 + l - 12}\cdot  \log(\pi) \} > \log(\pi).$
\end{tiny-remark}

\begin{tiny-remark} \label{1-rmk: The log-volume of ramification dataset2}
(i) For each $p\notin S_0$, write $e_0(p) = e_0 l$. For each $p\in S_0$, put
$$e_0(p) \defeq  \max_{e\in S_p^{\good}\cup S_p^{\multi}}\{e\},\; d_0(p) = \max_{e\in S_p^{\good}\cup S_p^{\multi}}(1+v_p(e)).$$
In addition, recall that for a proposition $P$, the Iverson Bracket is defined by $[P]_{\text{IB}} = 1$ if $P$ is true, and $[P]_{\text{IB}} = 0$ if $P$ is false. 

Then by a similar approach to the proof of Proposition \ref{1-prop:Estimates the log-volume of initial Theta-data}, we can show that 
{\small
\begin{equation}  \label{1.11-uppper bound of B_2(p)}
\begin{aligned}
B_2(p) \cdot \log(p) \le & [e_0(p) \ge p-1]_{\text{IB}} \cdot ( \frac{1}{p-1}+1-\frac{p-1}{e_0(p)}) \cdot \log(p) \\ & + [p\in S_0]_{\text{IB}} \cdot d_0(p) \cdot \log(p)  
  + [e_0(p) > p(p-1)]_{\text{IB}} \cdot \log(\frac{e_0(p)} {p-1} ) .
\end{aligned}
\end{equation} 
}

(ii) By a similar approach to the proof of Proposition \ref{1-prop:Estimates the log-volume of initial Theta-data}, we can also show that 
{\small
\begin{equation} \label{1.11-uppper bound of vol(R)}
\begin{aligned}
    \frac{l^2 + l - 12}{l^2 + 5l}\cdot \Vol(\mathfrak R)
    \le \log(\pi) + & \sum_{e_0(p) \ge p-1 } ( \frac{1}{p-1}+1-\frac{p-1}{e_0(p)}) \cdot \log(p) 
     \\&  +  \sum_{p\in S_0} d_0(p) \cdot \log(p)  
      + \sum_{e_0(p) > p(p-1)} \log(\frac{e_0(p)} {p-1} ) .
\end{aligned}
\end{equation} 
}

\end{tiny-remark}

\begin{tiny-remark}
It is worth noting that the value of $\Vol(\mathfrak R)$ can be computed explicitly by the algorithm in the proof of Corollary \ref{1-cor: The log-volume of ramification dataset}.
It is also worth noting that in the proof of each inequality in this section, the condition that 3-torsion points of $E_F$ are defined over $F$ is not used.
\end{tiny-remark}

\section{Construction of initial \texorpdfstring{$\Theta$}{}-Data}
In this section, we will prove a local version of effective abc-type inequalities associated to the triples $(a,b,c)$, where $a,b,c$ are non-zero coprime integers such that $$a+b=c.$$

\subsection{The general construction}
We shall provide a procedure for constructing $\mu_6$-initial $\Theta$-data from an elliptic curve defined over $\Q$.
Some results in the arithmetic of elliptic curves are needed for the construction.

\begin{prop} \label{2-prop: torsion point field}
Let $K$ be a number field, $E$ be an elliptic curve defined over $K$, 
$j(E)\in K$ be the $j$-invariant of $E$.
Let $w\in \V(L)$ be a nonarchimedean valuation of $L$ lying over some $v\in \V(K)^{\non}$, with residue characteristic $p$. 
Then:

(i) For any integer $n\ge 2$, the $n$-torsion point field $K(E[n])$ is Galois over $K$;
we have $\mu_n \subseteq K(E[n])$, where $\mu_n$ denotes the group of $n$-th roots of $1$ in $\overline{\Q}$; 
if $m,n\ge 2$ are coprime integers, then $K(E[mn]) = K(E[m],E[n])$, $K(E[m])\cap K(E[n])=K$.

(ii) Suppose that $l\ge 3$ is a prime number, $l\neq p$, 
and all the $l$-torsion points of $E$ are defined over $K_v$ [i.e. $K_v = K_v(E[l])$]. 
Then $E$ has semi-stable reduction at $v$.

(iii) Suppose that $E$ has good reduction at $v$. 
If $l\neq p$, then the field extension $L_w/K_v$ is unramified;
if $l=p$ and $v$ is not ramified over $\Q$, 
then the ramification index of the field extension $L_w/K_v$ belongs to the set $\{l-1,l(l-1),l^2-1\}$.

(iv) Suppose that $E$ has bad multiplicative reduction at $v$, $n\ge2$ is an integer. 
Write $q_K$ for the $q$-parameter of $E_{K_v}\defeq E\times_K K_v$,
$e_v$ for the ramification index of the field extension $K(E[n])/K$ at $v$.
Then we have $v(q_K) = -v(j(E))$, and there exists an unramified field extension $L/K_v$ of degree 1 or 2,
such that $L(E[n]) = L(\mu_n,q_K^{1/n})$. 
Hence if $p \nmid n$, then the field extension $K_v(E[n])/K_v$ is tamely ramified, and $e_v = n/\gcd(v(q_K),n)$; 
if $v_p(n)=1$ and the field extension $K_v/\Q_p$ is unramified, then $e_v = (p-1) \cdot n/\gcd(v(q_K),n)$.

(v) The field of moduli [cf., e.g., \cite{AbsTopIII}, Definition 5.1, (ii)] of $E$ is the field generated over $\Q$ by the $j$-invariant $j(E)$ of $E$. 
If there exists a prime number $l\ge 3$, such that all the $l$-torsion points of $E$ are defined over $K$, 
then any model of $E_{\overline{K}} \defeq E\times_K \overline{K}$ over $K$ is isomorphic to $E$, hence any model of $E_{\overline{K}}$ over its field of moduli is also a model of $E$.
\end{prop}
\begin{proof}
For assertion (i), cf. \cite{TorsionPointFields}, Proposition 5.2.1 and Proposition 5.2.2. 
For assertion (ii), cf. \cite{IUTchIV}, Proposition 1.8 (v).
For assertion (iii), cf. \cite{Bosch1990PropertiesON}, \S7.4 Theorem 5 and the introduction of \cite{Serre1971PropritsGD}.
Assertion (iv) follows from the theory of Tate curves, cf. \cite{Silverman1994AdvancedTI}, \S5 for a reference. 
Assertion (v) follows from \cite{IUTchIV}, Proposition 1.8 (ii), (iv).
\end{proof}

\begin{prop} \label{2-prop: elliptic curves admitting cores}
Let $E$ be an elliptic curve defined over a number field $K$,
$X$ be the punctured elliptic curve [defined over $K$] associated to $E$.
Suppose that $$j(E)\notin \{0,2^6\cdot 3^3,2^2\cdot 73^3\cdot 3^{-4},2^{14}\cdot 31^3\cdot 5^{-3}\},$$ then $X$ admits a $K$-core [cf. \cite{CanLift}, remark 2.1.1].
\end{prop}
\begin{proof}
The proposition follows immediately from \cite{Sijs}, Table 4 [cf. also \cite{Sijs}, Lemma 1.1.1; \cite{CanLift}, Proposition 2.7].
\end{proof}

\begin{prop} \label{2-prop: rational isogeny}
Let $E$ be an elliptic curve defined over $\Q$ with $j$-invariant $j(E)\in \Q$, $l$ be a prime number. 
Suppose that (a) $l\ge 23$ and $l\neq 37,43,67,163$; or (b) $E$ is semi-stable and $l\ge 11$; or (c) $l\ge 11$, $l\neq 13$ and the denominator of $j(E)$ is not a power of $2$.
Then $E$ doesn't admit a $\Q$-rational isogeny of degree $l$.
\end{prop}
\begin{proof}
The proposition follows immediately from \cite{Mazur1978RationalIO}, Theorem 1, Theorem 4 and Corollary 4.4.
\end{proof}

\begin{cor}\label{2-cor: mod-l representation}
Let $E$ be an elliptic curve defined over $\Q$, $l$ be a prime number. 
Let $p$ be a prime number, such that $v_p(j(E)) < 0$ and $l\nmid v_p(j(E))$.
Suppose that (a)  $l\ge 23$ and $l\neq 37,43,67,163$; or (b) $E$ is semi-stable and $l\ge 11$; or (c) $l\ge 11$, $l\neq 13$ and the denominator of $j(E)$ is not a power of $2$.
Then the the mod-$l$ Galois representation $\rho=\rho_{E,l}:G_{\Q}\to\GL(2,\F_l)$ associated to $E$ is a surjection.
\end{cor}
\begin{proof}
Write $H$ for the image $\rho(G_{\Q}) \subseteq \GL(2,\F_l)$. By taking into account that the determinant of the image of $\rho$ is $\F_l^\times$, we only need to prove that $\SL(2,\F_l)\subseteq H$.

Since $v_p(j(E)) < 0$, $E$ has potentially bad multiplicative reduction at $p$.
Hence there exists a $p$-adic local field $k$, such that $[k:\Q_p]\mid 2$, and $E_k\defeq E\times_{\Q} k$ has bad multiplicative reduction. Since $l\nmid v_p(j(E))$, it follows from the discussion of the local theory preceding \cite{GenEll}, Lemma 3.2 that the image $H$ contains [up to conjugacy] the element $\alpha=\big( \begin{smallmatrix}1&1\\0&1 \end{smallmatrix} \big) \in\GL(2,\F_l)$.

Since (a) or (b) or (c) is true, by Lemma \ref{2-prop: rational isogeny}, $E$ doesn't admit a $\Q$-rational isogeny of degree $l$. 
Hence $H$ is not contained in the Borel subgroup  
$\big( \begin{smallmatrix}*&*\\0&* \end{smallmatrix} \big) \subseteq\GL(2,\F_l)$, so $H$ contains at least one matrix that is not upper triangular. Then by \cite{GenEll}, Lemma 3.1 (iii), we have $\SL(2,\F_l)\subseteq H$.
\end{proof}

Now we shall start the construction of $\mu_6$-initial $\Theta$-data.
\begin{prop} \label{2-prop: construction of mu_6 initial Theta-data}
Let $E$ be an elliptic curve defined over $\Q$ with $j$-invarient $j(E) \in \Q$;
$N$ be the denominator of $j(E)$;
$F$ be a number field which is Galois over $\Q$;
$l\ge 5$ be a prime number such that $l\nmid [F:\Q]$;
$X_F$ be the punctured elliptic curve associated to $E_F \defeq E \times_{\Q} F$.
Suppose that:
\begin{itemize}
\item[(a)] $\sqrt{-1}\in F$, $F(E[6]) = F$, $E_F$ is semi-stable, and $F \subseteq \Q(E[n])$ for some positive integer $l\nmid n$. 
\item[(b)] $j(E)\notin \{0,2^6\cdot 3^3,2^2\cdot 73^3\cdot 3^{-4},2^{14}\cdot 31^3\cdot 5^{-3}\}$.
\item[(c)] We have $l\ge 23$ and $l\neq 37,43,67,163$; or $E$ is semi-stable and $l\ge 11$; or $l\ge 11$, $l\neq 13$ and $N$ is not a power of $2$.
\item[(d)] We have $N'_l \neq 1$, where $N'_l \defeq \prod_{p: p\neq l, l\nmid v_p(N)} p^{v_p(N)}$.
\end{itemize} 
Then there exists a $\mu_6$-initial $\Theta$-data $$\mathfrak D(E,F,l,\mu_6) = (\overline{F}/F,X_F,l,\underline{C}_K, \uV,\V_{\tmod}^{\bad}, \underline{\epsilon}),$$ 
which is of type $(l, N, N'_l)$ [cf. the notation in Definition \ref{1-def: ramification dataset for initial theta-data}, (i)].
\end{prop}
\begin{proof}
The proof will be shown by making use of Lemma \ref{1-lem: initial Theta-data}.

(1) Write $F_{\tmod} \defeq \Q$, then $F$ is Galois over $F_{\tmod}$, $l \nmid [F:F_{\tmod}]$.
By (a), we have $\sqrt{-1}\in F$, $6$-torsion points of $E_F$ are rational over $F$, and $E_F$ is semi-stable.
Let $\overline{F}$ be an algebraic closure of $F$, write $G_F \defeq \Gal(\overline{F}/F)$ for the absolute Galois group of $F$.

(2) Let $X$ be the punctured elliptic curve defined over $\Q$ associated to $E$, then $X_F = X \times_{\Q} F$. Hence the field of moduli of $X_F$ is $\Q = F_{\tmod}$. 
By (a), the elliptic curve $E_F$ is semi-stable.
Write \begin{equation*}
\V_{\tmod} \defeq \V(F_{\tmod}) = \V_{\Q}, 
\; \V_{\tmod}^{\bad} \defeq \{v_p \in \V_{\Q}^{\non}: p\neq l,\; l \nmid v_p(N) \}.
\end{equation*}
Then since $N'_l \neq 1$ by (d), $\V_{\tmod}^{\bad}$ is a nonempty set of nonarchimedean valuations of $F_{\tmod}$.

(3) For each $v\in \V(F)$ lying over some $v_p \in \V_{\tmod}^{\bad}$, we have $p\neq l$ and $l\nmid v_p(N) = v_p(q_p)$, where $q_p$ is the $q$-parameter of $E$ at $v_p$.
Write $e'$ for the ramification idex of $F_v$ over $\Q_p$. Since $l\nmid [F:\Q]$, $e' \mid [F:\Q]$, we have $l\nmid e'$. Then since $l\nmid  v_p(N) = v_p(q_p)$, we have $l \nmid  e'\cdot v_p(q_p) = v(q_p)$.
Hence for each $v\in \V(F)^{\bad}\defeq \V_{\tmod}^{\bad}\times_{\V_{\tmod}} \V(F)$, $l$ is prime to the residue characteristics of $v$, as well as to the order of the $q$-parameter of $E_F$ at $v$.

(4) By Corollary \ref{2-cor: mod-l representation} and (c), the mod-$l$ Galois representation $\rho_{E,l}:G_{\Q}\to\GL(2,\F_l)$ associated to $E$ is a surjection. 
Note that since $F\subseteq \Q(E[n])$ and $l\nmid n$ by (a), we have $F\cap \Q(E[l]) = \Q$ by Lemma \ref{2-prop: torsion point field}, (i).
Hence the image of the mod-$l$ Galois representation $\rho_{E_F,l}$ [associated to $E_F$] equals the image of $\rho_{E,l}$ [which equals $\GL(2,\F_l)$],
thus $\rho_{E_F,l}$ is a surjection. 

(5) Write $K\defeq F(E_F[l])\subseteq \overline{F}$ for the $l$-torsion point field of $E_F$. Then since $F$ and $\Q(E[l])$ are Galois over $F_{\tmod} = \Q$, the field $K = F \cdot \Q(E[l])$ is a finite Galois extension of $F_{\tmod}$.
Write $X_K \defeq X_F\times_F K$, then by Proposition \ref{2-prop: elliptic curves admitting cores} and (b), $X_K$ admits a $K$-core.

Now by the above facts (1)--(5), the existence of the $\mu_6$-initial $\Theta$-data $\mathfrak D(E,F,l,\mu_6)$ follows from Lemma \ref{1-lem: initial Theta-data},
and $\log(\mathfrak q) = \log(N'_{l})$ follows easily from Definition \ref{1-def:Log of arithmetic divisors}.
\end{proof}

\subsection{The consturction for Frey-Hellegouarch curves}
Let $(a,b,c)$ be a triple of non-zero coprime integers such that $a+b=c.$
The Frey-Hellegouarch curve associated to $(a,b,c)$ is the elliptic curve defined over $\Q$ by the equation $$y^2=x(x-a)(x+b).$$
This curve is popularized in its application to Fermat's Last Theorem, where one investigates a (hypothetical) solution to the Fermat equation $x^l + y^l = z^l$.

\begin{lem} \label{2-lem: properties of the F-H curve}
Let $(a,b,c)$ be a triple of non-zero coprime integers such that $a+b=c$;
$E$ be the Frey-Hellegouarch curve associated to $(a,b,c)$, which is defined over $\Q$ by the equation $y^2=x(x-a)(x+b).$
Let $l\ge 5$ be a prime number.

(i) By Tate's algorithm, we have
\begin{equation*}
c_4(E)=16(a^2+ab+b^2), \; \Delta(E)=16a^2b^2c^2, 
\; j(E)=\frac{256(a^2+ab+b^2)^3}{a^2b^2c^2}\in\Q.
\end{equation*}

For each prime number $p\neq 2$, the Frey-Hellegouarch curve $E$ has good reduction at $v_p$ when $p \nmid abc$, and has multiplicative reduction at $v_p$ when $p\mid abc$.
Moreover, if $4\mid (a+1)$ and $16\mid b$, then $E$ is semi-stable.

(ii) Write $F\defeq \Q(\sqrt{-1},E[3])$, then $F$ is Galois over $\Q$, $F=F(E[6])$, $E_F$ is semi-stable, and $F\subseteq \Q(E[12])$.
Write $$N \defeq \frac{|abc|}{\gcd(16,abc)}, \; N'_l \defeq \prod_{p: p\neq l, l\nmid v_p(N)} p^{v_p(N)}.$$
Then $N^2$ is the denominator of $j(E)$, $v_2(N) = \max\{v_2(abc)-4,0\}$, and we have $v_p(N)=v_p(abc)$ for each prime number $p\ge 3$.
Write $X_F$ for the punctured elliptic curve associated to $E_F\defeq E\times_F F$.

(iii) Suppose that $N'_l \neq 1$ and $(|a|,|b|,|c|)$ is not a permutation of $(1,1,2)$, $(1,8,9)$. 
Also, suppose that (a) $l\ge 11$ and $l\neq 13$; 
or (b) $l\ge 11$, $4\mid (a+1)$ and $16\mid b$.

Then by Proposition \ref{2-prop: construction of mu_6 initial Theta-data}, there exists a $\mu_6$-initial $\Theta$-data 
$$\mathfrak D(E,F,l,\mu_6) = (\overline{F}/F,X_F,l,\underline{C}_K, \uV,\V_{\tmod}^{\bad}, \underline{\epsilon}),$$ which is of type $(l, N^2, (N'_l)^2)$.

(iv) In the situation of (iii), proceeding with the notation in \S1.
For each $v_p\in \V_{\tmod}^{\non}$, write $e_p$ for the ramification index of $K_{\eta(v)}$ over $\Q_p$, $d_p$ for the different index of $K_{\eta(v)}$.
Then:
\begin{itemize}
\item When $p\mid abc$ and $p\neq 2$, $E$ has multiplicative reduction at $v_p$.
In this case, if $p\neq 3,l$, then $e_p = 6l/\gcd(v_p(N^2),6l)=3l/\gcd(v_p(abc),3l)$ divides $3l$, $p\nmid e_p$, $d_p=1-\frac{1}{e_p} \le 1-\frac{1}{3l}$;
if $p \in \{3, l\}$, then $e_p = 3l(p-1)/\gcd(v_p(abc),3l)$, hence 
$e_3\in\{2, 6, 2l, 6l\}$, $e_l \in \{l-1, 3(l-1), l(l-1), 3l(l-1)\}$, $d_3 \le 2$, $d_l \le 2$.

\item When $p\mid abc$ and $p\neq 2$, $E$ has good reduction at $v_p$.
In this case, if $p\neq 3,l$, then $e_p = 1$, $d_p = 0$; if $p \in \{3, l\}$, then $e_3 \in \{2,6,8\}$, $e_l \in \{l-1, l(l-1), l^2-1\}$, $d_3 \le 2$, $d_l \le 2$.

\item Note that $v_2(abc)\ge 1$. When $v_2(abc)\ge 5$, $E' \defeq E\times_{\Q} \Q(\sqrt{-1})$ has multiplicative reduction at the unique valuation of $\Q(\sqrt{-1})$ lying over $v_2$, then $e_2 = 6l/\gcd(v_2(abc)-4,3l)$, $e_2\in\{2,6,2l,6l\}$, $d_2 \le 2$;
when $v_2(abc) = 4$, $E$ has good reduction at $v_2$, then $e_2 = 2$, $d_2\le 2$;
when $1\le v_2(abc)\le 3$, $E$ has potentially good reduction at $v_2$, we have $2\mid e_2$, $e_2\mid 48$ and $d_2\le 1+v_2(e_2)\le 5$.
\end{itemize}

\end{lem}
\begin{proof}
For assertion (i), the case $p\neq 2$ follows from Tate's algorithm.
When $4\mid (a+1)$ and $16\mid b$,
by the change of variables $x\mapsto 4x$, $y\mapsto 8y+4x$, $E$ can be defined by the new equation 
\begin{align*}
E':y^2+xy=x^3+\frac{b-a-1}{4}\cdot x^2-\frac{ab}{16}\cdot x.
\end{align*}
By Tate's algorithm, we can show that 
\begin{align*}
c_4(E')=a^2+ab+b^2,\quad \Delta(E')=2^{-8}a^2b^2c^2,\quad j(E') = \frac{(a^2+ab+b^2)^3}{2^{-8}a^2b^2c^2}.
\end{align*}
Then since $\gcd(c_4(E'),\Delta(E')) = 1$, $E'$ is semi-stable, so $E$ is semi-stable.

Assertion (ii) is a consequence of Proposition \ref{2-prop: torsion point field} and Proposition \ref{2-prop: elliptic curves admitting cores}.

For assersion (iii), since $(|a|,|b|,|c|)$ is not a permutation of $(1,1,2)$, $(1,8,9)$, we have $j(E)\notin \{0,2^6\cdot 3^3,2^2\cdot 73^3\cdot 3^{-4},2^{14}\cdot 31^3\cdot 5^{-3}\}$, and $N$ is not a power of $2$.
Then assersion (iii) follows easily from Proposition \ref{2-prop: construction of mu_6 initial Theta-data}.

Assertion (iv) can be proved via Lemma \ref{1-lem: Estimates of Differents}, Proposition \ref{2-prop: torsion point field} and assertions (i), (ii). 
We shall prove the case of $p=2$, $1\le v_2(abc) \le 3$ as an example.
In this case, $v_2(N) = 0$, thus $E$ has potentially good reduction at $v_2$.
Since $\sqrt{-1}\in F \subseteq K$, we have $2\mid e_2$.
Let $v'\in \V(F)$ be the image of $\uv \defeq \eta(v_2)$ via the surjection $\V(K)\twoheadrightarrow \V(F)$.
By Proposition \ref{2-prop: torsion point field}, $E_F$ has good reduction at $v'$, thus the extension $K_{\uv}/F_{v'}$ is unramified.
Then since $\Q\subseteq \Q(\mu_3) \subseteq F = \Q(\sqrt{-1}, E[3])$ and $\Q(\mu_3)/\Q$ is unramified at $v_2$,
we have $e_2 \mid [F:\Q(\mu_3)] \mid 48$, and hence $d_2\le 1+v_2(e_2) \le 5$ by Lemma \ref{1-lem: Estimates of Differents}.
\end{proof}

\begin{prop} \label{2-prop: local abc-inequalities}
Let $(a,b,c)$ be a triple of non-zero coprime integers such that $a+b=c$, $l\ge 11$ be a prime number. Write $N \defeq |abc| / \gcd(16,abc)$.
Suppose that (a) $l\neq 13$; 
or (b) $16\mid abc$.

Let $\mathfrak{R}_l$ be the ramification dataset consists of the following data:
\begin{gather*}
l_0 = l,\; e_0 = 3,\;  S_0=\{2,3,l\},\; S_{\gen}^{\multi} = \{1,3,l,3l\},\;
S_2^{\good} = \{e: e\mid 48, 2\mid e\}, \\
S_3^{\good} = \{2,6,8\},\; S_l^{\good} = \{l-1, l(l-1), l^2-1\},\;
S_2^{\multi} = \{2,6,2l,6l\}, \\
S_3^{\multi} = \{2,6,2l,6l\},\; S_l^{\multi} = \{l-1, 3(l-1), l(l-1), 3l(l-1)\}.
\end{gather*}
Then with $\Vol(\mathfrak{R}_l)$ defined in Corollary \ref{1-cor: The log-volume of ramification dataset}, we have
{\small
\begin{equation} \label{2-1-eq0}
\begin{aligned}
\log(N) \le (3+\frac{11l+31}{l^2+l-12}) \cdot \log\rad(N) + 3\Vol(\mathfrak{R}_l) 
+ \sum_{p:\, p=l \,\text{or}\, l\mid v_p(N)} v_p(N)\cdot\log(p).
\end{aligned}
\end{equation}}

\end{prop}
\begin{proof}
Write $N'_l \defeq \prod_{p: p\neq l, l\nmid v_p(N)} p^{v_p(N)}.$
Then it suffices to show that
{\small
\begin{equation} \label{2-1-eq1}
\begin{aligned}
\frac{1}{3}\log(N'_l) \le (1+\frac{11l+31}{3l^2+3l-36}) \cdot \log\rad(N) + \Vol(\mathfrak{R}_l).
\end{aligned}
\end{equation}}

When $N'_l = 1$ or $(|a|,|b|,|c|)$ is a permutation of $(1,1,2)$, $(1,8,9)$, we have 
 $\frac{1}{3}\log(N'_l) < \log(3) < \log(\pi)$. Then (\ref{2-1-eq1}) follows from $\Vol(\mathfrak{R}_l) > \log(\pi) > \frac{1}{3}\log(N'_l)$ by Remark \ref{1-rmk: The log-volume of ramification dataset}. Hence we can assume that $N'_l \neq 1$ and $(|a|,|b|,|c|)$ is not a permutation of $(1,1,2)$, $(1,8,9)$.
 
 In the case (b) where $16\mid abc$, since (\ref{2-1-eq1}) is symmetric for $(a,b,c)$, we can assume that $16\mid b$. Then if $4\nmid a$, we can replace $(a,b,c)$ by $(-a,-b,-c)$, thus we can assume that $4\mid (a+1)$. Hence we can assume that (a) $l\neq 13$; or (b) $4\mid (a+1)$ and $16\mid b$.
 
Then by Lemma \ref{2-lem: properties of the F-H curve}, (iii), there exists a $\mu_6$-initial $\Theta$-data $\mathfrak D = \mathfrak D(E,F,l,\mu_6)$, which is of type $(l, N^2, (N'_l)^2)$.
 Moreover, by Lemma \ref{2-lem: properties of the F-H curve}, (iv) and Definition \ref{1-def: ramification dataset for initial theta-data}, $\mathfrak D$ admits $\mathfrak R_l$. Then (\ref{2-1-eq1}) follows directly from Corollary \ref{1-cor: The log-volume of ramification dataset}.
\end{proof}

\begin{tiny-remark} \label{2-rmk: remark of local abc-inequalities}
When we have $16\mid abc$ in Proposition \ref{2-prop: local abc-inequalities},
we can define the ramification dataset $\mathfrak{R}'_l$ by definding $S_2^{\good} = \{2\}$ instead in the defintion of $\mathfrak{R}_l$. Then similar to the proof of Proposition \ref{2-prop: local abc-inequalities}, 
we have
{\small
\begin{equation*} 
\begin{aligned}
\log(N) \le (3+\frac{11l+31}{l^2+l-12}) \cdot \log\rad(N) + 3 \Vol(\mathfrak{R}'_l)
+ \sum_{p:\, p=l \,\text{or}\, l\mid v_p(N)} v_p(N)\cdot\log(p).
\end{aligned}
\end{equation*}}

\end{tiny-remark}

\subsection{Effective partial abc inequalities}

We shall call (\ref{2-1-eq0}) the partial abc inequality. 
In this subsection,  we will prove an effective version of  (\ref{2-1-eq0}) by estimating the term $\Vol(\mathfrak R_l)$.
The following lemma in analytic number theory is needed, where the ``sum over $p$'' runs over prime numbers.

\begin{lem} \label{2-lem: new estimates of arithmetic functions}
(i) For real number $x>3$, put
{\small\begin{equation*}
\begin{aligned}
 f_1(x) &\defeq \sum_{p\le x}\frac{\log(p)}{p-1}+ (1+\frac{1}{x}) \cdot \sum_{p\le x}\log(p)
 -\frac{1}{x}\sum_{p\le x} p\cdot\log(p) + \sum_{p<\sqrt{x}+1 } \log(\frac{x} {p-1} )
 \\ f_2(x) &\defeq \frac{x^2+5x}{x^2+x-12}\cdot(f_1(3x) + \frac{10}{3}\cdot \log(x) + 9 ).
\end{aligned}
\end{equation*}}

Then for integer $n\ge 2 \cdot  10^5$, we have 
\begin{align*}
 f_2(n) < \frac{3}{2}\cdot n + 0.06 \cdot \frac{n}{\log(n)} .
\end{align*}

(ii) For real number $x\ge 1$, write 
$$S_x \defeq \{p\in\Primes: \sqrt{x/\log(2)}  < p \le \frac{2}{3}\sqrt{x \log(x)}\},$$
$$f_3(x) \defeq |S_x|,\quad f_4(x) = \sum_{p\in S_x }p.$$
Then for any real number $ x \ge e^{31}$, we have 
\begin{align*}        
 f_3(x)\ge \frac{4}{3}\cdot\frac{\sqrt{x}\cdot(\sqrt{\log(x)}-2.14) }{\log(x)+\log\log(x)},
 \quad f_4(x) < \frac{4}{9}x\cdot(1+\frac{4.33}{\log(x)}).
\end{align*}

\end{lem}
\begin{proof}
Write $\pi(x)\defeq \sum_{p\le x}1$ for the prime counting function and write $\vartheta(x)\defeq \sum_{p\le x}\log(p)$ for the first Chebyshev function.

First, we consider assertion (i). We shall use the Stieltjes integral.

By \cite{Dusart1999TheKP}, Section 4, we have
{\small
\begin{equation} \label{2-3-eq1}
	|\vartheta(x)-x| \le  \epsilon\cdot\frac{x}{\log(x)} \quad \text{for} \; x\ge A,
\end{equation}}
where $A= 2.89\cdot 10^7$, $\epsilon=0.006788$.
Then by (\ref{2-3-eq1}), for $x\ge A$  we have 
{\small
\begin{equation*}
\begin{aligned}
 &\sum_{A < p\le x} p\cdot\log(p) 
 = \int_{A}^{x} t\cdot d(\vartheta(t))
 = [t\cdot \vartheta(t)]\big|_{A}^{x} - \int_{A}^{x} \vartheta(t) dt
 \\ \ge & x\cdot\vartheta(x)-A\cdot\vartheta(A) 
 - \int_{A}^{x} (t+\epsilon\cdot\frac{\log(A)}{2\log(A)-1}\cdot\frac{2t\log(t)-t}{\log^2(t)}) dt
 \\ \ge & x\cdot\vartheta(x)-A\cdot\vartheta(A) 
 - \int_{A}^{x} (t+0.515\epsilon\cdot\frac{2t\log(t)-t}{\log^2(t)}) dt
 \\ \ge & (1-\frac{\epsilon}{\log(x)})x^2-(1+\frac{\epsilon}{\log(A)})A^2
 - [\frac{t^2}{2}+\frac{0.515 \epsilon t^2}{\log(t)}]\big|_A^x
 \\ \ge & (\frac{1}{2}-\frac{1.515 \epsilon}{\log(x)}) x^2
 - (\frac{1}{2}+\frac{ 0.485 \epsilon}{\log(A)}) A^2.
\end{aligned}
\end{equation*}}
For $x\ge 10 \cdot A$, the computation in \cite{code_of_zpzhou} shows that
{\small
\begin{equation*}
\sum_{p\le A} p\cdot\log(p) - (\frac{1}{2}+\frac{0.485 \epsilon}{\log(A)}) A^2 
> -\frac{0.008 \epsilon \cdot (10 \cdot A)^2}{\log(10 \cdot A)}
\ge -\frac{0.008 \epsilon x^2}{\log(x)}.
\end{equation*}}
Hence for $x\ge 10 \cdot A$ we have
{\small 
\begin{equation} \label{2-3-eq2}
\begin{aligned}
 \sum_{p\le x} p\cdot\log(p)         
 & = \sum_{A < p\le x} p\cdot\log(p) + \sum_{p\le A} p\cdot\log(p)
 \\ & \ge (\frac{1}{2}-\frac{1.523 \epsilon}{\log(x)}) x^2-\frac{0.008 \epsilon x^2}{\log(x)}
 = (\frac{1}{2}-\frac{1.523 \epsilon}{\log(x)}) x^2.
\end{aligned}
\end{equation}}

By \cite{Approximate_formulas_for_some_functions_of_prime_numbers}, Equation (3.24), we have  
\begin{align*}
  \sum_{p\le x}\frac{\log(p)}{p} < \log(x) \quad \text{for} \; x>1.
\end{align*}
Thus by the computation in \cite{code_of_zpzhou}, we have
{\small
\begin{equation} \label{2-3-eq3}
\begin{aligned}
 \sum_{p\le x} \frac{\log(p)}{p-1}      
 &= \sum_{p\le x} \frac{\log(p)}{p} + \sum_{p\le x} \frac{\log(p)}{p(p-1)}
 < \log(x) + \sum_{p\le A} \frac{\log(p)}{p(p-1)} + \sum_{n>A} \frac{\log(n)}{n(n-1)} 
 \\ &< \log(x) + \sum_{p\le A} \frac{\log(p)}{p(p-1)} + \frac{1}{\sqrt{A}} 
 < \log(x) + 0.8. 
\end{aligned}
\end{equation}}

By \cite{Approximate_formulas_for_some_functions_of_prime_numbers}, Equation (3.14), (3.29), for $x \ge 10\cdot A > \exp(19.48)$, 
we have $\vartheta(\sqrt{x}) > \sqrt{x}\cdot (1-\frac{1}{\log(x)} ) > 0.94 \sqrt{x}$, 
and $\sum_{p\le \sqrt{x}+1 } \log(\frac{p}{p-1}) \le \log( e^\gamma \log(\sqrt{x}+1)(1+\frac{1}{2\log(\sqrt{x}+1)} ) ) < \log\log(\sqrt{x}+1) + 0.63 < \log\log(x)$, where $\gamma = 0.57721\dots$ is Euler's constant. 
Hence for $x \ge 10\cdot A$, we have
{\small
\begin{equation} \label{2-3-eq4}
\begin{aligned}
 \sum_{p < \sqrt{x}+1 } \log(\frac{1}{p-1}) \le \sum_{p\le \sqrt{x}+1 } \log(\frac{p}{p-1})  - \vartheta(\sqrt{x}) 
 < \log\log(x) - 0.94 \sqrt{x}.
\end{aligned}
\end{equation}
}

By \cite{Dusart2018ExplicitEO}, Corollary 5.2,  for $x > 1$ we have
{\small
\begin{equation} \label{2-3-eq5}
\begin{aligned}
\frac{x}{\log(x)} \mathop{\le}\limits_{x\ge 17} \pi(x) 
 \mathop{\le}\limits_{x>1} \frac{x}{\log(x)}(1+\frac{1}{\log(x)}+\frac{2.53816}{\log^2(x)})
 \mathop{\le}\limits_{x\ge e^{31}} \frac{x}{\log(x)}(1+\frac{1.082}{\log(x)}) .
\end{aligned}
\end{equation}
}
Then for $x \ge 10\cdot A > \exp(19.48)$, we have 
$\sum_{p < \sqrt{x}+1 } \log(x) \le (\pi(\sqrt{x})+1) \cdot \log(x) 
\le 2\sqrt{x}(1+\frac{1}{\log(\sqrt{x})}+\frac{2.53816}{\log^2(\sqrt{x})})+\log(x) < 2.26\sqrt{x}$.
Hence by (\ref{2-3-eq4}), $\sum_{p<\sqrt{x}+1 } \log(\frac{x} {p-1} ) 
= \sum_{p < \sqrt{x}+1 } \log(x) + \sum_{p < \sqrt{x}+1 } \log(\frac{1}{p-1})  < 2.26\sqrt{x} + \log\log(x) - 0.94\sqrt{x} 
= 1.32 \sqrt{x} + \log\log(x) .$
Then by (\ref{2-3-eq1}), (\ref{2-3-eq2}) and (\ref{2-3-eq3}), for $x \ge 10\cdot A$ we have
{\small
\begin{equation*}
\begin{aligned}
 f_1(x) &= \sum_{p\le x}\frac{\log(p)}{p-1}+ (1+\frac{1}{x})\vartheta(x)
 -\frac{1}{x}\sum_{p\le x} p\cdot\log(p) + \sum_{p<\sqrt{x}+1 } \log(\frac{x} {p-1} )
 \\ &< \log(x)+0.8 + (1+\frac{1}{x})\cdot x\cdot (1+\frac{\epsilon}{\log(x)}) 
 - \frac{1}{x}\cdot(\frac{1}{2}-\frac{1.523 \epsilon}{\log(x)}) \cdot x^2 +  1.32 \sqrt{x} + \log\log(x) 
 \\ &< \frac{1}{2} \cdot x + 0.01865 \cdot \frac{x}{\log(x)} .
\end{aligned}
\end{equation*}}

The computationin \cite{code_of_zpzhou}] shows that $f_1(n) < \frac{1}{2} \cdot n + 0.01865 \cdot \frac{n}{\log(n)}$ holds for any integer $2\cdot 10^5 \le n \le 10\cdot A$. Hence for any integer $n \ge 2\cdot 10^5$,
we have 
{\small
\begin{equation*}
 f_1(n) < \frac{1}{2} \cdot n + 0.01865 \cdot \frac{n}{\log(n)} .
\end{equation*}
}
Then for any integer $n \ge 2\cdot 10^5$, we have
{\small
\begin{equation*}
\begin{aligned}
 \\ f_2(n) &= \frac{n^2+5n}{n^2+n-12}\cdot(f_1(3n) + \frac{10}{3}\cdot \log(n) + 9 )
 \\ &< (1+\frac{5}{n})\cdot (\frac{3}{2}\cdot n+ 0.05595 \cdot \frac{n}{\log(n)}
 + \frac{10}{3}\cdot \log(n) + 9)
 < \frac{3}{2}\cdot n + 0.06 \cdot \frac{n}{\log(n)} .
\end{aligned}
\end{equation*}
}
This proves assertion (i).

Next, we consider assertion (ii).
For $x\ge e^{31}$, we have $\log(x)\ge 31$.
Then by (eq5) we have 
\small\begin{align*}
 f_3(x) &= \pi(\frac{2}{3}\sqrt{x \log(x)}) - \pi(\sqrt{x/\log(2)})
 \\ &\ge \frac{\frac{2}{3}\sqrt{x \log(x)}}{\log(\frac{2}{3}\sqrt{x \log(x)})} - \frac{\sqrt{x/\log(2)}}{\log(\sqrt{x/\log(2)})}(1+\frac{1.082}{\log(\sqrt{x/\log(2)})})
 \\ &\ge \frac{\frac{4}{3}\sqrt{x \log(x)}}{\log(x)+\log\log(x)} 
 - \frac{2\sqrt{x/\log(2)}}{\log(x)+\log\log(x)}\cdot (1+\frac{\log\log(x)}{\log(x)}) \cdot (1+\frac{1.082\times 2}{\log(x)-\log\log(2)})
 \\ &\ge \frac{\frac{4}{3}\sqrt{x \log(x)}}{\log(x)+\log\log(x)} 
 - \frac{\sqrt{x}\cdot 2 / \sqrt{\log(2)} }{\log(x)+\log\log(x)}
 \cdot (1+\frac{\log(31)}{31}) \cdot (1+\frac{1.082\times 2}{31-\log\log(2)})
 \\ & > \frac{\frac{4}{3}\sqrt{x}\cdot(\sqrt{\log(x)}-2.14) }{\log(x)+\log\log(x)}.
\end{align*} \normalsize

To estimate $f_4(x)$, we shall write $\alpha = \sqrt{x/\log(2)}$ and $\beta = \frac{2}{3}\sqrt{x \log(x)}$ for convenience.
For $x \ge e^{31}$, by (eq5) we have
\small\begin{align*}
 f_4(x) &= \sum_{\alpha < p \le \beta} p
 = \int_{\alpha}^{\beta} p\cdot d(\pi(t))
 = \beta\cdot \pi(\beta)-\alpha\cdot \pi(\alpha) - \int_{\alpha}^{\beta}\pi(t)\cdot dt
 \\ &< \beta\cdot \pi(\beta) - \alpha\cdot \pi(\alpha) - \frac{1}{2}\int_{\alpha}^{\beta} d(\frac{t^2}{\log(t)})
 \\ &= \beta\cdot \pi(\beta) - \frac{\beta^2}{2\log(\beta)}
 - \alpha\cdot \pi(\alpha) + \frac{\alpha^2}{2\log(\alpha)}
 \\ &< \frac{\beta^2}{2\log(\beta)}(1+\frac{2.164}{\log(\beta)})
 < \frac{4}{9}x\cdot(1+\frac{4.33}{\log(x)}).
\end{align*}\normalsize
This proves assertion (ii).

\end{proof}

\begin{prop} \label{2-prop: the 1st ABC inequality}
Let $(a,b,c)$ be a triple of non-zero coprime integers such that $a+b=c$, $l\ge 11$ be a prime number. Write $N \defeq |abc| / \gcd(16,abc)$.
Suppose that (a) $l\neq 13$; or (b) $16\mid abc$.
Then with $f_2(l)$ defined in Lemma \ref{2-lem: new estimates of arithmetic functions}, (i), 
we have
{\small
\begin{equation*}  
\begin{aligned}
\log(N) \le \, & (3+\frac{11l+31}{l^2+l-12}) \cdot \log\rad(N)
+ \sum_{p:\, p=l \,\text{or}\, l\mid v_p(N)} v_p(N)\cdot\log(p)  + 3f_2(l).
\\ \mathop{\le}_{n\ge 2\cdot 10^5} &  (3+\frac{11l+31}{l^2+l-12}) \cdot \log\rad(N)
+ \sum_{p:\, p=l \,\text{or}\, l\mid v_p(N)} v_p(N)\cdot\log(p) + \frac{9}{2}\cdot n + 0.18 \cdot \frac{n}{\log(n)} . 
\end{aligned}
\end{equation*}
}
\end{prop}
\begin{proof}
Proceeding with the notation of Proposition \ref{2-prop: local abc-inequalities}, we shall estimate the upper bound of $\Vol(\mathfrak{R}_l)$ via Remark \ref{1-rmk: The log-volume of ramification dataset2}.
Recall that $\mathfrak{R}_l$ consists of the following data:
\begin{gather*}
l_0 = l,\; e_0 = 3,\;  S_0=\{2,3,l\},\; S_{\gen}^{\multi} = \{1,3,l,3l\},\;
S_2^{\good} = \{e: e\mid 48, 2\mid e\}, \\
S_3^{\good} = \{2,6,8\},\; S_l^{\good} = \{l-1, l(l-1), l^2-1\},\;
S_2^{\multi} = \{2,6,2l,6l\}, \\
S_3^{\multi} = \{2,6,2l,6l\},\; S_l^{\multi} = \{l-1, 3(l-1), l(l-1), 3l(l-1)\}.
\end{gather*}
Following the notation of Remark \ref{1-rmk: The log-volume of ramification dataset2},
for $p\neq 2,3,l$, we have $e_0(p) = 3l$; for $p\in\{2,3,l\}$, we have $e_0(2) = \max\{48,  6l \} \le 6l$, $d_0(2) = 5$, $e_0(3) = 6l$, $d_0(3)=2$, $e_0(l) = 3l(l-1)$, $d_0(l) = 2$.
Then by (\ref{1.11-uppper bound of vol(R)}) we have
\begin{equation} \label{2-2-eq0}
\begin{aligned}
\frac{l^2 + l - 12}{l^2 + 5l} \cdot \Vol(\mathfrak R_l)
     \le \log(\pi) + & \sum_{p \le e_0(p)+1 } (\frac{1}{p-1}+1-\frac{p-1}{e_0(p)}) \cdot \log(p) 
      \\&  +  \sum_{p\in S_0} d_0(p) \cdot \log(p)  
       + \sum_{e_0(p) > p(p-1)} \log(\frac{e_0(p)} {p-1} ) .
\end{aligned}
\end{equation} 
We shall estimate the partial sums in the right hand of (\ref{2-2-eq0}).

For $p \in \{2,3,l\}$, we have $\frac{1}{p-1}+1-\frac{p-1}{e_0(p)} < (\frac{1}{p-1}+1-\frac{p-1}{3l}) + \frac{p-1}{3l}$.
For $p\notin \{2,3,l\}$, $e_0(p) = 3l$. If we have $p \le e_0(p)+1 = 3l+1$, then since $p$ is a prime number, we have $p\le 3l-2$. 
Hence we have
{\small
\begin{equation} \label{2-2-eq1}
\begin{aligned}
&\sum_{p \le e_0(p)+1 } (\frac{1}{p-1}+1-\frac{p-1}{e_0(p)}) \cdot \log(p) 
\\\le&  \sum_{p\le 3l-2}(\frac{1}{p-1}+1-\frac{p-1}{3l})\cdot\log(p) 
+ \frac{\log(2)+2\log(3)+(l-1)\log(l)}{3l} . 
\end{aligned}
\end{equation}
}

Since $S_0 = \{2,3,l\}$, we have 
{\small
\begin{equation} \label{2-2-eq2}
\sum_{p\in S_0} d_0(p) \cdot \log(p)
= d_0(2) \cdot \log(2)  + d_0(3) \cdot \log(3)  + d_0(l) \cdot \log(l)  
= \log(2^5\cdot 3^2\cdot l^2). 
\end{equation}
}

For $p\notin \{2,3,l\}$, $e_0(p) = 3l$. If we have $p(p-1) < e_0(p) = 3l$, then $p < \sqrt{3l}+1 < l$. Since $e_0(2)  = 6l$, $e_0(3) = 6l$, $e_0(l) = 3l(l-1)$, we have 
{\small
\begin{equation} \label{2-2-eq3}
\begin{aligned}
 \sum_{e_0(p) > p(p-1)} \log(\frac{e_0(p)} {p-1} )
\le & \log(\frac{e_0(2)}{3l}) + \log(\frac{2e_0(3)}{3l}) + \sum_{p < \sqrt{3l}+1 } \log(\frac{3l} {p-1} ) + \log(\frac{e_0(l)}{l-1})
\\ <& \sum_{p < \sqrt{3l}+1 } \log(\frac{3l} {p-1} ) + \log(24l) .
\end{aligned}
\end{equation}
}

Finally, since $l\ge 11$, by combaining (\ref{2-2-eq0}), (\ref{2-2-eq1}), (\ref{2-2-eq2}), (\ref{2-2-eq3}) and Lemma \ref{2-lem: new estimates of arithmetic functions}, (i), we have
{\small
\begin{equation*}  
\begin{aligned}
 & \Vol(\mathfrak R_l)
 <  \frac{l^2+5l}{l^2+l-12}\cdot \big(f_1(3l) + \log(\pi)  + \frac{\log(2)+2\log(3)+(l-1)\log(l)}{3l} + \log(2^5\cdot 3^2\cdot l^2) 
 \\ & + \log(20l) \big)  <  \frac{l^2+5l}{l^2+l-12}\cdot \big(f_1(3l) + \frac{10}{3}\cdot \log(l) + 9\big) = f_2(l) \mathop{<}\limits_{n\ge 2\cdot 10^5} \frac{3}{2}\cdot n + 0.06 \cdot \frac{n}{\log(n)} .
\end{aligned}
\end{equation*}
}

\end{proof}

\section{Effective abc inequalities}
In this section, we focus on effective versions of abc inequalities over the rational number field. 

\begin{lem}\label{3-lem:ABC-type inequality}
Let $N$ be a positive integer, $h\defeq \log(N).$ 
Let $S$ be a finite set of prime numbers $\ge 5$, with cardinality $n\defeq |S|\ge 2$.
Let $p_0$ be the smallest prime number in $S$.

Let $2 \le k \le n$ be an integer, $k(S)$ be the the product of the smallest $k$ numbers in $S$. 
Suppose that $k(S) > h/\log(2)$.

For each prime number $l$, write $N_l\defeq \prod_{p:l\mid v_p(N)}p^{v_p(N)}$.
For each $l\in S$, suppose that
\small\begin{align} \label{3-1-eq1}
    \sum_{p:p\neq l, l\nmid v_p(N)}v_p(N)\cdot \log(p)
    \le (3+\frac{11l+31}{l^2+l-12}) \cdot\log\rad(N) + 3\Vol(l). 
\end{align}\normalsize
for some real number $\Vol(l) \ge 0$.

(i) Write
\small\begin{align*}
    A\defeq \{p\in S : p\mid N\},
    \; B\defeq \{p\mid N: \exists l\in S,\, l\mid v_p(N) \},
    \; C\defeq \{p\mid N: p \notin (A\cup B)\};
\end{align*}\normalsize
\small\begin{align*}
    N_A \defeq \prod_{p\in A}p^{v_p(N)}, 
\; N_B \defeq \prod_{p\in B}p^{v_p(N)}, 
\; N_C \defeq \prod_{p\in C}p^{v_p(N)}.
\end{align*}\normalsize
Then we have
\small\begin{align*}
\sum_{p\in S}v_p(N)\cdot\log(p) = \log(N_A),
\; \sum_{p\in S} \log(N_p)\le (k-1)\cdot\log(N_B).
\end{align*}\normalsize

(ii) Write 
\small\begin{align*}
    a_1 \defeq \frac{1}{n}\sum_{l\in S}\frac{11l+31}{l^2+l-12},
    \quad a_2 \defeq \frac{3}{n}\sum_{l\in S}\Vol(l),
    \quad a_3 \defeq \sum_{l\in S}\log(l).
\end{align*}\normalsize
Then we have 
\small\begin{align*}
    h \le (3+a_1)\cdot\log\rad(N) + a_2 + \frac{kh}{n},
\end{align*}\normalsize

(iii) With the notaion in assertion (ii), we have
\small\begin{align*}
    h \le (3+a_1)\cdot\log\rad(N_C) + (\frac{k}{n}+\frac{3+a_1}{p_0})(h- \log(N_C) ) + a_2 + (3+a_1)a_3.
\end{align*}\normalsize

\end{lem}
\begin{proof}
First, we consider assertion (i). 
By the definition of the set $A$, we have 
\small\begin{align*}
    \sum_{p\in S}v_p(N)\cdot\log(p) = \sum_{p:p\in S,p\mid N}v_p(N)\cdot\log(p) = \sum_{p\in A}v_p(N)\cdot\log(p) = \log(N_A).
\end{align*}\normalsize
By the definition of $N_p$, We have
\small\begin{align*}
    \sum_{p\in S} \log(N_p)
    &= \sum_{p \in S} \sum_{l:l\mid N,p\mid v_l(N)} v_l(N)\cdot \log(l)
    = \sum_{p \in S} \sum_{l:l\in B,p\mid v_l(N)} v_l(N)\cdot \log(l)
    \\ &= \sum_{l\in B}v_l(N)\cdot \log(l)\cdot\big(\sum_{p:p \in S,p\mid v_l(N)} 1 \big).
\end{align*}\normalsize

Now for each prime number $l\mid N$, we claim that $\sum_{p:p\in S, p\mid v_l(N)} 1 \le k-1$, which can be proved by contradiction as follows.
Assume that $p_1,\dots,p_k\in S$ are $k$ distinct prime numbers, such that $p_i\mid v_l(N)$ for $1\le i\le k$.
Then $\prod_{1\le i\le k}p_i \ge k(S) > h/\log(2)$ and $\prod_{1\le i\le k}p_i  \mid v_l(N)$.
Hence
\begin{align*}
    h = \log(N) &\ge v_l(N)\cdot \log(l)\ge \prod_{1\le i\le k}p_i\cdot\log(2)
    > (h/\log(2))\cdot\log(2) = h.
\end{align*}
--- a contradiction. Thus the claim is true. 

By the claim we have 
\small\begin{align*}
    \sum_{p\in S} \log(N_p)
    &= \sum_{l\in B}v_l(N)\cdot \log(l)\cdot\big(\sum_{p:p \in S,p\mid v_l(N)} 1 \big)
    \\&\le (k-1)\cdot\sum_{l\in B}v_l(N)\cdot \log(l) = (k-1)\cdot\log(N_B).
\end{align*}\normalsize
This proves assertion (i).

Next, we consider assertion (ii). By (\ref{3-1-eq1}), for each $l\in S$, we have 
\small\begin{align}  \label{3-1-eq1.1}
h = \log(N) \le (3+\frac{11l+31}{l^2+l-12}) \cdot\log\rad(N) 
+ 3 \Vol(l) + v_l(N)\cdot\log(l)+ \log(N_l) .
\end{align}\normalsize
Hence by taking the average of (\ref{3-1-eq1.1}) for $l\in S$ and by assertion (i), we have
{\small
\begin{equation} \label{3-1-eq2}
\begin{aligned}
 h &\le (3+a_1) \cdot\log\rad(N) + a_2 + \frac{1}{n}\cdot\log(N_A)+\frac{k-1}{n}\cdot\log(N_B)
    \\ &\le (3+a_1) \cdot\log\rad(N) + a_2 + \frac{kh}{n}.
\end{aligned}
\end{equation}
}
This proves assertion (ii).

Finally, we consider assertion (iii). By the definition of $N_A$, $N_B$, we have 
\small\begin{gather*}
\log\rad(N_A) = \sum_{p\in A}\log(p) \le \sum_{p\in S}\log(p) = a_3,
\\ \log(N_B) = \sum_{p\in B}v_p(N)\cdot\log(p)\ge \sum_{p\in B}p_0\cdot \log(p) = p_0\cdot\log\rad(N_B).
\end{gather*}\normalsize
Thus
\small\begin{align*}
    \log\rad(N) &\le \log\rad(N_A)+\log\rad(N_B)+\log\rad(N_C)
    \\ &\le \log\rad(N_C) + \frac{1}{p_0}\cdot \log(N_B) + a_3.
\end{align*}\normalsize
Then by (\ref{3-1-eq2}) and $\log(N_A),\log(N_B)\le \log(N)- \log(N_C) $, we have  
\small\begin{align*}
    h  &\le (3+a_1) \cdot\log\rad(N) + a_2 + \frac{1}{n}\cdot\log(N_A)+\frac{k-1}{n}\cdot\log(N_B)
    \\ &\le (3+a_1)\cdot \log\rad(N_C) + (3+a_1)a_3 + (\frac{k-1}{n}+\frac{3+a_1}{p_0})\log(N_B) + \frac{1}{n}\log(N_A) + a_2
    \\ &\le (3+a_1)\cdot\log\rad(N_C) + (\frac{k}{n}+\frac{3+a_1}{p_0})(\log(N)- \log(N_C) ) + a_2 + (3+a_1)a_3.
\end{align*}\normalsize
This proves assertion (iii).
\end{proof}

\begin{thm}\label{3-thm:ABC-type inequality}
Let $(a,b,c)$ be a triple of non-zero coprime integers such that $a+b=c$.

(i) Suppose that $\log(|abc|)\ge 700$, then we have
\begin{align*}
    \log(|abc|) \le 3\log\rad(abc) + 8 \sqrt{\log(|abc|)\cdot \log\log(|abc|)}.
\end{align*} 

(ii) Suppose that $\log(|abc|)\ge 3\cdot 10^{13}$, then we have
\begin{align*}
    \log(|abc|) \le 3\log\rad(abc) + 3&\sqrt{\log(|abc|)\cdot \log\log(|abc|)}
    \\ &\cdot (1 + \frac{6}{\sqrt{\log\log(|abc|)}} + \frac{15}{\log\log(|abc|)}).
\end{align*} 
\end{thm}
\begin{proof}
Write $N \defeq |abc|/\gcd(16,abc)$, $h\defeq \log(N).$ 
We shall proceed with the notaion in Lemma \ref{3-lem:ABC-type inequality}.
We will choose suitable $S$, such that for each $l\in S$, we have $l\ge 11$ and $l\neq 13$. 
Then by Proposition \ref{2-prop: the 1st ABC inequality} we can take 
\begin{align*}
\Vol(l)=f_2(l) \mathop{\le}_{n\ge 2\cdot 10^5} \frac{3}{2}\cdot l + 0.06 \cdot \frac{l}{\log(l)}.
\end{align*} 

For assertion (i), consider the case $h \ge e^{31}$ at first. 
We shall write $$k=2,\quad S=\{l\in P: \sqrt{h/\log(2)} < l \le \frac{2}{3}\sqrt{h \log(h)} \}.$$
Since $h\ge e^{31}$, for any $l\in S$, we have $l \ge p_0 > \sqrt{h/\log(2)} > e^{15}$.
Since $k(S) > p_0^2 > h/\log(2)$ and $n = |S| \ge 2$, 
the set $S$ satifies the premise of Lemma \ref{3-lem:ABC-type inequality}.

Since $\log\rad(N)\le h$, by Lemma \ref{3-lem:ABC-type inequality}, (ii) we have 
\begin{align}  \label{3-2-eq1}
    h \le 3\cdot\log\rad(N) + a_1\cdot h+a_2 + \frac{2h}{n}, 
\end{align}
where $$a_1 = \frac{1}{n}\sum_{l\in S}\frac{11l+31}{l^2+l-12},
\quad a_2 = \frac{3}{n}\sum_{l\in S}\Vol(l).$$
Since $p_0 > e^{15}$, by Proposition \ref{2-prop: the 1st ABC inequality} we can take 
$$\Vol(l)=f_2(l) < \frac{3}{2}\cdot l + 0.06 \cdot \frac{l}{\log(l)}
\le \frac{3}{2}\cdot l + \frac{0.06 \cdot \frac{2}{3}\sqrt{h \log(h)}}{\log(\frac{2}{3}\sqrt{h \log(h)})} 
\le \frac{3}{2}\cdot l + \frac{0.08 \sqrt{h}}{ \sqrt{\log(h)} }
$$ for each $l\in S$.
Also, by Lemma \ref{2-lem: new estimates of arithmetic functions}, (ii) we have
\small\begin{align*}
f_3(h) = n = |S| \ge \frac{\frac{4}{3}\sqrt{h}\cdot(\sqrt{\log(h)}-2.14) }{\log(h)+\log\log(h)},
\; f_4(h) = \sum_{l\in S}l < \frac{4}{9}h\cdot(1+\frac{4.33}{\log(h)}).
\end{align*}\normalsize
Hence for $h\ge e^{31}$, we have
\small\begin{align*}
    \\ a_1 &= \frac{1}{n}\sum_{l\in S}\frac{11l+31}{l^2+l-12}
    < \frac{1}{n}\sum_{l\in S}\frac{12}{l} < \frac{12}{p_0}<\frac{12\sqrt{\log(2)}}{\sqrt{h}} < \frac{10}{\sqrt{h}}
    \\ a_2 &= \frac{3}{n}\sum_{l\in S}\Vol(l)
    < \frac{9}{2n}\sum_{l\in S}l + 3 \cdot \frac{0.08 \sqrt{h}}{ \sqrt{\log(h)} }
    = \frac{9f_4(h)}{2n} + \frac{0.24 \sqrt{h}}{ \sqrt{\log(h)} },
\end{align*}\normalsize
and
\small\begin{align*}   
    \frac{9f_4(h)}{2n}+\frac{2h}{n}
    &< \frac{4h}{n}\cdot (1+\frac{2.165}{\log(h)}) 
    = \frac{4h}{f_3(h)}\cdot (1+\frac{2.165}{\log(h)}) 
    \\ &< 3\sqrt{h\log(h)}\cdot(1+\frac{2.165}{\log(h)})\cdot (1-\frac{2.14}{\sqrt{\log(h)}})^{-1}\cdot(1+\frac{\log\log(h)}{\log(h)}).
\end{align*}\normalsize
Thus we have
{\small
\begin{equation} \label{3-2-eq2}
    a_1\cdot h+a_2 + \frac{2h}{n} 
    < 10\sqrt{h} + \frac{0.24 \sqrt{h}}{\sqrt{\log(h)} } + \big(\frac{9f_4(h)}{2n}+\frac{2h}{n}\big)
    < 3\sqrt{h\log(h)}\cdot g(\log(h)),
\end{equation}
}
where for $x=\log(h)\ge 31$, we have
\small\begin{align*}
    g(x) \defeq (1+\frac{2.165}{x})\cdot(1-\frac{2.14}{\sqrt{x}})^{-1}\cdot(1+\frac{\log(x)}{x}) + \frac{10}{3\sqrt{x}} + \frac{0.08}{x} 
    < 1 + \frac{6}{\sqrt{x}} + \frac{14.5}{x}.
\end{align*}\normalsize
Then for $h\ge e^{31}$, by (\ref{3-2-eq1}), (\ref{3-2-eq2}) we have 
\small\begin{align*}
    h - 3\log\rad(N)&\le a_1\cdot h+a_2 + \frac{2h}{n} 
    < 3\sqrt{h\log(h)}\cdot g(\log(h))
    \\ &<3\sqrt{h\log(h)}\cdot (1 + \frac{6}{\sqrt{\log(h)}} + \frac{14.5}{\log(h)}),
\end{align*}\normalsize
hence for $h\ge e^{31}$, we have 
\small\begin{align} \label{3-2-eq3}
    h - 3\log\rad(N) + 4\log(2) 
    < 3\sqrt{h\log(h)}\cdot (1 + \frac{6}{\sqrt{\log(h)}} + \frac{15}{\log(h)})
    < 8\sqrt{h\log(h)} . 
\end{align}\normalsize
The computation in \cite{code_of_zpzhou} shows that for $680 \le h \le e^{31}$, 
we have \small\begin{align}  \label{3-2-eq4}
    h - 3\log\rad(N) + 4\log(2) < 8 \sqrt{h\log(h)}.
\end{align}\normalsize
Then (\ref{3-2-eq4}) is true for any $h \ge 680$ by (\ref{3-2-eq3}).

Since $N=|abc|/\gcd(abc,16)$, we have 
$h=\log(N)\le \log(|abc|) \le h+4\log(2).$
Hence when $\log(|abc|)\ge 700$, we have $h\ge \log(|abc|)-4\log(2) > 680$. 
Then by (\ref{3-2-eq4}), we have
\small\begin{align*}
    \log(|abc|) - 3\log\rad(abc) 
    \le& h - 3\log\rad(N) + 4\log(2) 
    < 8 \sqrt{h\log(h)} 
    \\ \le& 8 \sqrt{\log(|abc|)\cdot \log\log(|abc|)}.  
\end{align*}\normalsize
This proves assertion (i).

For assertion (ii), note that we have 
$h \ge \log(|abc|)-4\log(2) \ge 3\cdot 10^{13}-4\log(2) > e^{31}$. 
Then assertion (ii) follows easily from (\ref{3-2-eq3}).
\end{proof}

\begin{cor}\label{3-cor:ABC-type inequality, eps-version}
Let $(a,b,c)$ be a triple of non-zero coprime integers such that $a+b=c$;
$\epsilon$ be a positive real number $\le \frac{1}{10}$. Then we have 
\begin{align}\label{3-3-eq0}
    |abc| \le \max\{
    \exp\left(400 \cdot \epsilon^{-2} \cdot \log(\epsilon^{-1} ) \right),
    \rad(abc)^{3+3\epsilon} \}.
\end{align}
\end{cor}
\begin{proof}
Write \begin{align*}
h_0  \defeq 400 \cdot \epsilon^{-2} \cdot \log(\epsilon^{-1} ), 
\quad u_0 \defeq 64 \cdot (1 + \epsilon^{-1})^2 .
\end{align*}
For real number $x \ge 10$, write 
\begin{align*}
f(x) \defeq  &\log(400 \cdot x^{2} \cdot \log(x)) - \log\log(400 \cdot x^{2} \cdot \log(x) ) 
- \log( 64 \cdot (1 + x)^2).
\end{align*}
For $x\ge 10$, the derivative
\begin{align*}
f'(x) = \frac{2}{x(x+1)} + \frac{\log(400\log(x))-1}{x\log(x)\log(400 \cdot x^{2} \cdot \log(x))} > 0.
\end{align*}
Then since $\epsilon^{-1} \ge 10$, we have $$f(\epsilon^{-1}) = \log(h_0) - \log\log(h_0) - \log(u_0) \ge f(10) > 0.$$  
Hence we have
\begin{align} \label{3-3-eq1}
\frac{h_0}{\log(h_0)} > u_0.
\end{align}

When $|abc|\le e^{h_0} = \exp\left(400 \cdot \epsilon^{-2} \cdot \log(\epsilon^{-1} ) \right)$, (\ref{3-3-eq0}) is true. 
Hence we can assume that $|abc|>  e^{h_0}.$

Write $h \defeq \log(|abc|) > h_0 > 40000$, 
then  by (\ref{3-3-eq1}) we have 
\begin{align} \label{3-3-eq2}
    \frac{h}{\log(h)} > \frac{h_0}{\log(h_0)} > u_0.
\end{align}
By Theorem \ref{3-thm:ABC-type inequality}, (i), we have 
\begin{align*}
    h \le 3 \log\rad(|abc|) + 8 \sqrt{h\cdot \log(h)}.
\end{align*}
Then by (\ref{3-3-eq2}) and the definition of $u_0$, we have  
\begin{align*}
    \frac{3 \log\rad(|abc|)}{h} &\ge 1 - 8 \sqrt{\frac{\log(h)}{h}}
    \ge 1-\frac{8}{\sqrt{u_0}} = \frac{1}{1+\epsilon},
\end{align*}
which implies
$$h = \log(|abc|) \le (3+3\epsilon)\cdot \log\rad(|abc|).$$
Hecne $$|abc| \le \rad(abc)^{3+3\epsilon}.$$
\end{proof}

\begin{tiny-remark}
With the same notation, Corollary \ref{3-cor:ABC-type inequality, eps-version} improves the following inequality in \cite{ExpEst}, Theorem 5.4:
\begin{align*}
|abc| \le 2^4 \cdot \max\{\exp(1.7 \cdot 10^{30} \cdot \epsilon^{-166/81}), \rad(abc)^{3+3\epsilon}\}.
\end{align*}
\end{tiny-remark}

\section{Generalized Fermat Equations}

Let $r,s,t\ge 2$ be positive integers. The equation
\begin{equation} \label{GenFE} \tag{$\star$}
x^r + y^s = z^t, \;\text{with}\; x, y, z\in\Z
\end{equation}
is known as the generalized Fermat equation with signature $(r, s, t)$.
A solution $(x, y, z)$ of (\ref{GenFE}) is called non-trivial if $xyz\neq 0$, primitive if $\gcd(x, y, z)=1$, and positive if $x, y, z \in \N$.

Finding all the non-trivial primitive solutions to (\ref{GenFE}) is a long-standing problem in number theory. 
The Pythagorean triples, which are the positive solutions to (\ref{GenFE}) with signature $(2, 2, 2)$, have been known since ancient times. 
Fermat's Last Theorem, which is proven by Andrew Wiles in 1994 \cite{Wiles},
states that (\ref{GenFE}) has no positive primitive solution with signature $(n, n, n), n\ge 3$.
Catalan's conjecture, which is proven by Preda Mih\u{a}ilescu in 2002 \cite{Mihilescu2004PrimaryCU}, states that if $x=1$ in (\ref{GenFE}), then the only positive integer solution to $1 + y^s = z^t$ (with $s, t\ge 2$) is $(y, z, s, t) = (2, 3, 3, 2)$.

\subsection{A brief survey}
The behavior of primitive solutions to (\ref{GenFE}) is fundamentally determined by the size of the quantity $$\chi(r, s, t) = \frac{1}{r} + \frac{1}{s} + \frac{1}{t} - 1.$$
Here $\chi$ is the Euler characteristic of a certain stack associated to (\ref{GenFE}). 

For the spherical case where  $\chi > 0$, $(r, s, t)$ is a permutation of $(2, 2, t)$, $t\ge 2$, or $(2, 3, t)$, $3\le t\le 5$. 
In this case, parametrizations for non-trivial primitive solutions for these signatures can be found in Beukers \cite{spherical_case1}, Edwards \cite{spherical_case2} and Cohen \cite{Cohen-GTM240}, Section 14.

For the Euclidean case where $\chi = 0$, $(r, s, t)$ is a permutation  of $(2, 3, 6)$, $(2, 4, 4)$ or $(3, 3, 3)$.  
In this case, non-trivial primitive solutions for these signatures come from the solution $1^6 + 2^3 = 3^2$, cf. \cite{Bennett-Chen-Dahmen-Yazdani}, Proposition 6 for a proof.

For the hyperbolic case where $\chi < 0$, the number of primitive solutions for each fixed signature $(r, s, t)$ is finite, cf. Darmon-Granville \cite{Darmon-Granville}.
The following positive primitive solutions for the hyperbolic case are currently known:
\begin{gather*}
    1^n+2^3 = 3^2 \,\,(\text{for } n\ge 7), \quad 2^5+7^2 = 3^4, \quad 7^3+13^2 = 2^9, \quad
    2^7+17^3 = 71^2, \\
    3^5+11^4 = 122^2, \quad 17^7+76271^3 = 21063928^2, \quad
    1414^3+2213459^2 = 65^7, \\
    9262^3+15312283^2 = 113^7, \; \,
    43^8+96222^3 = 30042907^2, \; \, 33^8+1549034^2 = 15613^3.
\end{gather*}
Since all known solutions have $\min\{r, s, t\} \le 2$, Andrew Beal conjectured in 1993 (cf. \cite{Beal_Conj}) that (\ref{GenFE}) has no positive primitive solution when $r, s, t\ge 3$.
This conjecture is known as the Beal conjecture, also known as the Mauldin conjecture and the Tijdeman-Zagier conjecture.

We shall say a signature $(r,s,t)$ is \textbf{solved} if all the non-trivial primitive solutions to the generalized Fermat equation $x^r + y^s = z^t$ are known. 
Then each signature with $\chi(r,s,t) = \frac{1}{r}+\frac{1}{s}+\frac{1}{t}-1 > 0$ is solved.

The generalized Fermat equation (\ref{GenFE}) has been solved for many signatures with $\chi \le 0$, some of them are listed in the following Theorem \ref{4-thm:conclusions about the generalized Fermat equation}.

\begin{thm}\label{4-thm:conclusions about the generalized Fermat equation}
For the signatures $(r,s,t)$ listed below, the generalized Fermat equation has been solved, i.e. all the non-trivial primitive solutions to $x^r + y^s = z^t$ are related to the known solutions presented above.

\begin{itemize}
\setlength{\itemsep}{0pt}
\item $(n, n, n), n\ge 3$, cf. Wiles \cite{Wiles}, Taylor-Wiles \cite{Taylor-Wiles}, Mochizuki-Fesenko-Hoshi-Minamide-Porowski\cite{ExpEst}.

\item $(n, n, 2), n\ge 4$; $(n, n, 3), n\ge 3$, cf. Darmon-Merel \cite{Darmon-Merel}, Poonen \cite{Poonen}.

\item $(2,3,n), n\in\{7,8,9,10,15\}$, cf. Poonen-Schaefer-Stoll \cite{Poonen-Schaefer-Stoll}, Bruin \cite{Bruin-16,Bruin-19}, Brown \cite{Brown_case_2_3_10}, Siksek-Stoll \cite{Siksek-Stoll_case_2_3_15}.
Also for $n=11$, which is conditional on generalized Riemann hypothesis, cf. Freitas-Naskr\k{e}cki-Stoll \cite{Freitas-Naskrecki-Stoll-case_2_3_11}.

\item $(2, 4, n), n\ge 4$, cf. Ellenberg \cite{Ellenberg}, Bennett-Ellenberg-Ng \cite{Bennett-Ellenberg-Ng};
$(2, n, 4), n\ge 4$, Corollary of Bennett-Skinner \cite{Bennett-Skinner-8}, Bruin \cite{Bruin-18}.

\item $(2, 6, n), n\ge 3$, cf. Bennett-Chen \cite{Bennett-Chen};
$(2, n, 6), n\ge 3$, cf. Bennett-Chen-Dahmen-Yazdani \cite{Bennett-Chen-Dahmen-Yazdani}.

\item $(2, 2n, 3), 3\le n\le 10^7$ or $n\equiv 5 (\tmod 6)$, cf. Chen \cite{Chen-21}, Dahmen \cite{Dahmen-29}.
\item $(2n, 2n, 5), n\ge 2$, cf. Bennett \cite{Bennett}; $(2, 2p, 5)$, prime $p\ge 17$, $p\equiv 1 (\tmod 4)$, cf. Chen \cite{Chen-22}.

\item $(2, 2n, k), n \ge 2, k \in \{9, 10, 15\}$; $(3,3,2n), (3,6,n), (4,2n,3), n\ge 2$; $(2m, 2n, 3)$, $m\ge 2$ and $n\equiv 3 (\tmod 4)$; $(3, 3p, 2)$, prime $p\equiv 1 (\tmod 8)$, cf. Bennett-Chen-Dahmen-Yazdani \cite{Bennett-Chen-Dahmen-Yazdani-3, Bennett-Chen-Dahmen-Yazdani}.

\item $(3,3,n), 3\le n\le 10^9 $, cf. Chen-Siksek \cite{Chen-Siksek}.

\item $(3,4,5)$, $(5,5,7)$, $(5,7,7)$, $(5,5,19)$, cf. Siksek-Stoll \cite{Siksek-Stoll}, Dahmen-Siksek \cite{Dahmen-Siksek}.

\item $(p,p,q)$, primes $p \ge q\ge 5$; $(p,p,q)$, primes $p \ge 5$, $q > 3\sqrt{p\log_2(p)}$, cf. Bartolomé-Mih\u{a}ilescu \cite{semi_local_approximation_fermat_catalan_popular}.
\end{itemize}
\end{thm}

\begin{tiny-remark}
Note that if a signature $(r,s,t)$ is solved, then $(s,r,t)$ is solved; 
if $(r,s,t)$ is solved and  $2\nmid s$, then $(s,t,r), (t,s,r)$ are solved;
when $\chi(r,s,t)\le 0$, if $(r,s,t)$ is solved and $r',s',t'$ are positive multiples of $r, s, t$ respectively, then $(r', s', t')$ is solved. 
Hence many new solved signatures can be obtained from the solved signatures in Theorem \ref{4-thm:conclusions about the generalized Fermat equation}.
For more about generalized Fermat equations, cf. Bennett-Mih\u{a}ilescu-Siksek \cite{Bennett-Mihilescu-Siksek_survey} and Bennett-Chen-Dahmen-Yazdani \cite{Bennett-Chen-Dahmen-Yazdani}.
\end{tiny-remark}

\begin{tiny-remark}
Most of the results in Theorem \ref{4-thm:conclusions about the generalized Fermat equation} are quoted from the tables in the first section of \cite{Bennett-Chen-Dahmen-Yazdani}, including the case of signatures $(2,n,4), n\ge 4$. As explained by Michael Bennett, the case of signatures $(2,n,4), n\ge 4$ can be proved in the following way.

(1) The signature $(2,4,4)$ is solved by Fermat; Bruin \cite{Bruin-18} covers the signature $(2,5,4)$; Bennett-Chen \cite{Bennett-Chen} covers the signature $(2,6,4)$. Hence it suffices to prove for $n = 9$ or $n$ is a prime number $\ge 7$. 

(2) Suppose that $x,y,z$ are non-zero coprime integers such that $x^2 + y^n = z^4$.
Then we have $(z^2+x)(z^2-x) = y^n$, where $\gcd(z^2+x, z^2-x)$ is $1$ or $2$.

When $\gcd(z^2+x, z^2-x) = 1$, we can write $z^2 + x = y_1^n,\; z^2 - x = y_2^n$, where $y_1,y_2,z$ are non-zero coprime integers. 
Hence we have $2z^2 = y_1^n + y_2^n$, which has no solution for $n\ge 6$ by Bennett-Skinner \cite{Bennett-Skinner-8}, Theorem 1.1.

When $\gcd(z^2+x, z^2-x) = 2$, we can write $z^2 + \epsilon x = 2y_1^n,\; z^2 - \epsilon x = 2^{n-1}y_2^n$, where $y_1,y_2,z$ are non-zero coprime integers, $x,z$ are odd and $\epsilon \in \{ \pm 1 \}$.
Hence we have $z^2 = y_1^n + 2^{n-2}y_2^n$. 
By Catalan's conjecture, we can show that $|y_1y_2| \neq 1$. 
Then when $n\ge 7$ is a prime number, the equation  $z^2 = y_1^n + 2^{n-2}y_2^n$ has no solution by Bennett-Skinner \cite{Bennett-Skinner-8}, Theorem 1.2.

(3) We are now left to prove the case where $n=9$ and $y$ is even.
Since $x^2 - z^4 =  (-y^3)^3$ and $8 \mid y^3$, by comparing to the parametrizations in Cohen \cite{Cohen-GTM240}, Section 14.4.1, 
we can assume that $y^3 = 8st(s^3 - 16t^3)(s^3 + 2t^3)$, where $s,t$ are non-zero coprime integers such that $s$ is odd and $3\nmid(s-t)$. 

Note that $s,t, s^3 - 16t^3, s^3 + 2t^3$ are coprime with each other, except for $\gcd(s^3 - 16t^3, s^3 + 2t^3) \in \{1,3,9\}$. But since their product is a cube, we must have $\gcd(s^3 - 16t^3, s^3 + 2t^3) = 1$, and hence $s,t, s^3 - 16t^3, s^3 + 2t^3$ are cubes.

Write $s^3 + 2t^3 = r^3$, where $r,s,t$ are non-zero coprime integers. Then one can check that $(\frac{6t}{r-s},\frac{9(r+s)}{r-s})$ is a rational point on the elliptic curve $E: Y^2=X^3-27$.
$E$ has Mordell-Weil rank $0$, and only has the rational points corresponding to the point at infinity and $(X,Y)=(3,0)$. Hence we have $(\frac{6t}{r-s},\frac{9(r+s)}{r-s}) = (3,0)$ or $r-s = 0$, thus $(r,s,t) = \pm (1,-1,1), \pm (1,1,0)$. 
But then $y^3 = 8st(s^3 - 16t^3)(s^3 + 2t^3) \in \{0,17\}$, which is impossible. 
\end{tiny-remark}

\subsection{Upper bounds}

Let $r,s,t\ge 2$ be positive integers.
Throughout this subsection, we shall assume that $(x,y,z)$ is a positive primitive solution to the generalized Fermat equation $x^r+y^s=z^t$. 
We shall prove upper bounds for $\log(x^r y^s z^t)$.

Write  $a=x^r, b=y^s, c=z^t$, then $(a,b,c)$ is a triple of positive coprime integers such that $a+b=c$. 
We shall follow the notaion of Lemma \ref{3-lem:ABC-type inequality} and make some assumptions:

$(1)$ $N \defeq |abc|/\gcd(16,abc)$, $h\defeq \log(N).$
For each prime number $l$, $N_l\defeq \prod_{p:l\mid v_p(N)}p^{v_p(N)}$.

$(2)$ $S$ is a finite set of prime numbers $\ge 11$, with cardinality $n\defeq |S|\ge 2$;
$2 \le k \le n$ is an integer, $k(S)$ is the the product of the smallest $k$ numbers in $S$;
write $u_0 \defeq \min\{r,s,t\} \ge 2$, write $p_0$ for the smallest prime number in $S$.
When $u_0 < 4$, suppose that for any $13\notin S$.

$(3)$ For each $l\in S$, suppose that 
\begin{align} \label{4-2-eq:volume1}
    \sum_{p:p\neq l, l\nmid v_p(N)}v_p(N)\cdot \log(p)
    \le (3+\frac{11l+31}{l^2+l-12}) \cdot\log\rad(N) + 3\Vol(l). 
\end{align}
for some real number $\Vol(l) \ge 0$. The value of $\Vol(l)$ can be taken as the value of $\Vol(\mathfrak{R}_l)$ in Proposition \ref{2-prop: the 1st ABC inequality}; or $\Vol(\mathfrak{R}'_l)$ in Remark \ref{2-rmk: remark of local abc-inequalities} when $u_0 \ge 4$, since we have $v_2(abc) =  v_2(x^r y^s z^t) \ge  u_0 \ge 4$.

$(4)$ Write 
\small\begin{align*}
    A\defeq \{p\in S : p\mid N\},
    \; B\defeq \{p\mid N: \exists l\in S,\, l\mid v_p(N) \},
    \; C\defeq \{p\mid N: p \notin (A\cup B)\};
\end{align*}\normalsize
\small\begin{align*}
    N_A \defeq \prod_{p\in A}p^{v_p(N)}, 
\; N_B \defeq \prod_{p\in B}p^{v_p(N)}, 
\; N_C \defeq \prod_{p\in C}p^{v_p(N)};
\end{align*}\normalsize
\small\begin{align*}
    a_1 \defeq \frac{1}{n}\sum_{l\in S}\frac{11l+31}{l^2+l-12},
    \quad a_2 \defeq \frac{3}{n}\sum_{l\in S}\Vol(l),
    \quad a_3 \defeq \sum_{l\in S}\log(l).
\end{align*}\normalsize

\begin{lem} \label{4-lem: prepare for upper bound I}
Proceeding with the above notation, write 
$$r' \defeq \min\{r, s\},\; s' \defeq \max\{r, s\}, 
\; 0\le h_C\defeq  \log(N_C) \le h = \log(N). $$

(i) We have
\begin{align*}
\log\rad(N_C) \le \frac{1}{u_0} h_C +  \frac{4 \log(2)}{u_0} .
\end{align*}

(ii) If $t \le r'$, then 
\begin{align*}
	\log\rad(N_C) \le  \frac{1}{r'} (h_C-h) +(\frac{1}{2t}+\frac{1}{2r'})h + \frac{4 \log(2)}{u_0} .
\end{align*}

(iii) If $t = r'$, then
\begin{align*}
\log\rad(N_C) \le   \frac{1}{s'} (h_C-h) + (\frac{1}{t}+\frac{1}{s'}) \cdot \frac{2th}{3t-1} + \frac{4 \log(2)}{u_0} .
\end{align*}

(iv) If  $t \ge r'$, then 
\begin{align*}
\log\rad(N_C) \le \frac{1}{t}(h_C-h) + \max\{ \frac{1}{2t} + \frac{1}{2r'},  \frac{1}{3}(\frac{1}{r}+\frac{1}{s}+\frac{1}{t})  \} \cdot h + \frac{4\log(2)}{u_0} .
\end{align*}

\end{lem}
\begin{proof}
When $(a,b,c)$ is a permutation of $(1,8,9)$, 
we have $u_0 = r'=2$ and $\log\rad(N_C) \le \log(3) < 2 \log(2) = \frac{4\log(2)}{u_0}$, thus the lemma is true. 
Now assume that  $(a,b,c)$ is not a permutation of $(1,8,9)$, 
and assume without loss of generality that $r \le s$, then $r' = r$, $s' = s$.

Since $a+b=c$ and $(a,b,c)$ is not a permutation of $(1,8,9)$, 
we have $a, b, c, x, y, z \ge 2$ by Catalan's conjecture (now Mih\u{a}ilescu's theorem, cf. \cite{Mihilescu2004PrimaryCU}).
Hence \begin{align} \label{4-1-eq1}
 a,b < c < (abc)^{1/2} \le (16 N)^{1/2}.
\end{align}

Write  $v \defeq \max\{v_2(abc),4\}$, then $abc = 2^v N$.
Write  $N'_C \defeq 2^v N_C \le 16 N_C$ if $2 \mid N_C$, 
and write $N'_C \defeq N_C$ if $2 \nmid N_C$. 
Then for each $p\mid N_C$, we have $v_p(N'_C) = v_p(abc)$.
Hence we can assume that $N_C \mid N'_C = x_1^r y_1^s z_1^t \le 16 N_C $ for some $x_1\mid x, y_1\mid y, z_1\mid z$. Then 
\begin{align} \label{4-1-eq2}
\log\rad(N_C) \le \log(x_1y_1z_1), \; \log(N'_C) \le \log(16N_C) = h_C + 4\log(2) .
\end{align}

For assertion (i), by (\ref{4-1-eq2})  we have
\begin{align*}
\log\rad(N_C) \le \log(x_1y_1z_1)  
\le \frac{1}{u_0 } \log(x_1^r y_1^s z_1^t)
= \frac{1}{u_0 } \log(N'_C) 
\le \frac{1}{u_0} h_C +  \frac{4 \log(2)}{u_0} .
\end{align*}

For assertion (ii), since $t\le r' = r \le s$, we have $u_0 = t$. 
Then by (\ref{4-1-eq1}), (\ref{4-1-eq2})  we have
\begin{align*}
    \log\rad(N_C)& \le \log(x_1y_1z_1) 
    = \frac{1}{r}\log(x_1^r y_1^r z_1^t)
    + (\frac{1}{t}-\frac{1}{r})\log(z_1^t)
    \\&\le \frac{1}{r}\log(N'_C) + (\frac{1}{t}-\frac{1}{r})\log(c)
    \le \frac{1}{r}\log(16N_C)+(\frac{1}{2t}-\frac{1}{2r})\log(16N)
    \\&\le \frac{1}{r} h_C +(\frac{1}{2t}-\frac{1}{2r})\log(N)+\frac{4 \log(2)}{u_0}.
\end{align*}

For assersion (iii), since $t = r' = r \le s$, we have $u_0 = t$. 
Since $z = c^{1/t} > a^{1/t} = x > 1$, we have $2 \le x \le z-1$, and
\begin{align} \label{4-1-eq3}
	b = y^s = z^t - x^t \ge z^t - (z-1)^t > z^{t-1} > (xz)^{(t-1)/2}  = (ac)^{(t-1)/(2t)}.
\end{align} 
Then
\begin{align*}
	\log(x_1^r z_1^t) = \log(ac) < \frac{2t}{3t-1} \log(abc) \le  \frac{2t}{3t-1} \log(16 N).
\end{align*} 
Hence we have
\begin{align*}
\log\rad(N_C)& \le \log(x_1 y_ 1 z_1) 
= \frac{1}{s}\log(x_1^t y_1^s z_1^t) + (\frac{1}{t}-\frac{1}{s}) \log(x_1^t z_1^t)
\\&\le \frac{1}{s}\log(16 N_C) + (\frac{1}{t}-\frac{1}{s}) \cdot \frac{2t}{3t-1} \log(16 N)
\\&\le \frac{1}{s} h_C + (\frac{1}{t}-\frac{1}{s}) \cdot \frac{2th}{3t-1} + \frac{4 \log(2)}{u_0}.
\end{align*} 

For assertion (iv), let $\ell_a = \log(a) / \log(abc)$, $\ell_b = \log(b) / \log(abc)$, $\ell_c = \log(c) / \log(abc)$. Then $0 \le \ell_a,\ell_b \le \ell_c$ and $\ell_a + \ell_b + \ell_c = 1$. Hence $\ell_a, \ell_b \le \ell_c \le \frac{1}{2}$ and  $\ell_a + \ell_b \le \frac{2}{3}$.
Consider about the upper bound of $(\frac{1}{r}-\frac{1}{t})\ell_a + (\frac{1}{s}-\frac{1}{t})\ell_b$. 

If $\ell_a \le \ell_b$, then since $\frac{1}{r}-\frac{1}{t} > \frac{1}{s}-\frac{1}{t}$, 
we have $(\frac{1}{r}-\frac{1}{t})\ell_a + (\frac{1}{s}-\frac{1}{t})\ell_b \le ((\frac{1}{r}-\frac{1}{t}) + (\frac{1}{s}-\frac{1}{t}))\cdot \frac{\ell_a + \ell_b}{2} \le \frac{1}{3}(\frac{1}{r}+\frac{1}{s}-\frac{2}{t})$. 

Case 2: $\ell_a \ge \ell_b$. Then we have $\ell_b \le 1-2\ell_a$ and $\frac{1}{3}\le \ell_a \le \frac{1}{2}$. Hence $(\frac{1}{r}-\frac{1}{t})\ell_a + (\frac{1}{s}-\frac{1}{t})\ell_b \le
(\frac{1}{r}-\frac{1}{t})\ell_a + (\frac{1}{s}-\frac{1}{t})(1-2\ell_a)
= \frac{1}{3}(\frac{1}{r}+\frac{1}{s}-\frac{2}{t}) + (\frac{1}{r}+\frac{1}{t}-\frac{2}{s})(\ell_a-\frac{1}{3})$. 
If $ \frac{1}{r}+\frac{1}{t} \le \frac{2}{s}$, by taking $\ell_a = \frac{1}{3}$, we can get $(\frac{1}{r}-\frac{1}{t})\ell_a + (\frac{1}{s}-\frac{1}{t})\ell_b \le \frac{1}{3}(\frac{1}{r}+\frac{1}{s}-\frac{2}{t})$;
if $ \frac{1}{r}+\frac{1}{t} \ge \frac{2}{s}$, by taking $\ell_a = \frac{1}{2}$, we can get $(\frac{1}{r}-\frac{1}{t})\ell_a + (\frac{1}{s}-\frac{1}{t})\ell_b \le \frac{1}{2r}-\frac{1}{2t}$.

In conclusion, we have 
\begin{align*}
(\frac{1}{r}-\frac{1}{t})\ell_a + (\frac{1}{s}-\frac{1}{t})\ell_b 
\le \max\{\frac{1}{2r}-\frac{1}{2t}, \frac{1}{3}(\frac{1}{r}+\frac{1}{s}-\frac{2}{t})\} .
\end{align*}
Hence
\small \begin{align*}
\log\rad(N_C)& \le \log(x_1 y_ 1 z_1) 
= \frac{1}{t}\log(x_1^t y_1^s z_1^t) + (\frac{1}{r}-\frac{1}{t})\log(x_1^r) + (\frac{1}{s}-\frac{1}{t})\log(y_1^s)  
\\& \le \frac{1}{t}\log(16N_C)  + ( (\frac{1}{r}-\frac{1}{t})\ell_a + (\frac{1}{s}-\frac{1}{t})\ell_b ) \log(abc)
\\&\le \frac{1}{t}(h_C-h) + \max\{ \frac{1}{2r}+\frac{1}{2t},  \frac{1}{3}(\frac{1}{r}+\frac{1}{s}+\frac{1}{t})  \} + \frac{4\log(2)}{u_0} .
\end{align*} \normalsize

\end{proof}

\begin{tiny-remark} \label{4-2-rmk1}
Write $\overline{u} \defeq \frac{1}{3}(\frac{1}{r}+\frac{1}{s}+\frac{1}{t})$.
When $\frac{2}{s'}\le \frac{1}{r'}+\frac{1}{t}$, we have $\overline{u} \le \frac{1}{2r'}+\frac{1}{2t}$, hence by Lemma \ref{4-lem: prepare for upper bound I}, (ii) and (iv), we have 
\begin{align*}
	\log\rad(N_C) \le  \frac{h_C-h}{\max\{r',t\} } + (\frac{1}{2t}+\frac{1}{2r'})h + \frac{4 \log(2)}{u_0} .
\end{align*}
When $\frac{2}{s'}\ge \frac{1}{r'}+\frac{1}{t}$, we have $r'\le t$, and $\frac{1}{2r'}+\frac{1}{2t} \le \overline{u} \le \frac{2}{3r'}+\frac{1}{3t}$, hence by Lemma \ref{4-lem: prepare for upper bound I}, (ii) and (iv), we have 
\begin{align*}
	\log\rad(N_C) \le  \frac{h_C-h}{t } + (\frac{2}{3r'}+\frac{1}{3t}) h + \frac{4 \log(2)}{u_0} .
\end{align*}

\end{tiny-remark}


\begin{lem}\label{4-lem:upper bound I}
Proceeding with the above notation, and suppose that $k(S) > h/\log(2)$.
Recall that $r' = \min\{r, s\}$, $s' = \max\{r, s\}$, $u_0 = \min\{r, s, t\}$.
We shall define $b_1$ in any of the following cases:
\begin{itemize}
\item[(a)] When $\frac{2}{s'}\le \frac{1}{r'}+\frac{1}{t}$, write $b_1 \defeq\max\{ \frac{k}{n} + \frac{3+a_1}{p_0}, \,
\frac{1}{2}(\frac{k}{n} + \frac{3+a_1}{p_0}) + \frac{3+a_1}{2 u_0 }, \,
(3+a_1) \cdot (\frac{1}{2r'} + \frac{1}{2t}) \}.$

\item[(b)] When $\frac{2}{s'}\ge \frac{1}{r'}+\frac{1}{t}$, write $b_1 \defeq\max\{ \frac{k}{n} + \frac{3+a_1}{p_0}, \,
\frac{1}{3}(\frac{k}{n} + \frac{3+a_1}{p_0})+\frac{2(3+a_1)}{3r'}, \, (3+a_1) \cdot (\frac{2}{3r'}+\frac{1}{3t}) \}.$

\item[(c)] When $t = r'$, write $b_1 \defeq \max\{ \frac{k}{n} + \frac{3+a_1}{p_0}, \,
(\frac{k}{n} + \frac{3+a_1}{p_0}) \cdot \frac{t-1}{3t-1} + \frac{2(3+a_1)}{3t-1}, \,
 (3+a_1)\cdot \frac{2s' + (t-1) }{s'(3t-1)}  \}. $
\end{itemize}
We shall write $b_2 \defeq a_2 + (3+a_1) \cdot (a_3+\frac{4\log(2)}{u_0}).$ 
Then we have
\begin{align} \label{4-2-eq0}
	h \le b_1 \cdot h + b_2 .
\end{align}
Hence if $b_1 < 1$ and  $k(S)\cdot \log(2) > \frac{b_2}{1-b_1},$
then $h=\log(N)$ cannot belong to the interval 
$$(\frac{b_2}{1-b_1}, k(S)\cdot \log(2) ) .$$
\end{lem}

\begin{proof}
Since $k(S) > h/\log(2)$, we can make use of Lemma \ref{3-lem:ABC-type inequality}.
We shall only prove the case when $b_1$ is defined in case (a),
Other cases can be proved similarly.

Recall that we have $$0\le h_C\defeq  \log(N_C) \le h = \log(N). $$
By Lemma \ref{3-lem:ABC-type inequality} (iii),
we have 
\begin{align} \label{4-2-eq1}
	h \le (3+a_1)\cdot\log\rad(N_C) + (\frac{k}{n}+\frac{3+a_1}{p_0})(h- h_C ) + a_2 + (3+a_1)a_3.
\end{align}

When $0 \le h_C \le \frac{1}{2}h$, 
by (\ref{4-2-eq1}) and  Lemma  \ref{4-lem: prepare for upper bound I} (i),
we have
\begin{align} \label{4-2-eq2}
	h \le (3+a_1)\cdot \frac{h_C}{u_0} + (\frac{k}{n}+\frac{3+a_1}{p_0})(h- h_C ) + b_2.
\end{align}
The right hand of (\ref{4-2-eq2}) is linear in $h_C$, 
reaching its maximal value at $h_C=0$ or $h_C=\frac{1}{2}h$.
Hence by (\ref{4-2-eq2}) we have
\begin{align*}
	h \le \max\{ \frac{k}{n} + \frac{3+a_1}{p_0}, \frac{k}{2n} + \frac{3+a_1}{2p_0} + \frac{3+a_1}{2u_0} \}\cdot h + b_2 \le b_1\cdot h + b_2.
\end{align*}

When $\frac{1}{2}h \le h_C \le h$, by (\ref{4-2-eq1}) and Remark \ref{4-2-rmk1},
we have
\begin{align} \label{4-2-eq3}
h \le (3+a_1)( \frac{h_c-h}{\max\{r',t\}} + (\frac{1}{2t}+\frac{1}{2r'})h )
+  (\frac{k}{n}+\frac{3+a_1}{p_0})(h- h_C ) + b_2.
\end{align}
The right hand of (\ref{4-2-eq3}) is linear in $h_C$, 
reaching its maximal value at $h_C=\frac{1}{2}h$ or $h_C = h$.
Hence by (\ref{4-2-eq3}) we have
\begin{align*}
	h \le \max\{\frac{k}{2n} + \frac{3+a_1}{2p_0} + \frac{3+a_1}{2u_0},
	\frac{3+a_1}{2t} + \frac{3+a_1}{2r'}  \}\cdot h + b_2 \le b_1\cdot h + b_2.
\end{align*}

Hence when $b_1$ is defined in case (a), we must have (\ref{4-2-eq0}),
and then the lemma follows.
\end{proof}

\begin{thm} \label{4-thm-bounds}
Let $r,s,t\ge 3$ be positive integers.
Write
\begin{align*}
S(r,s,t) \defeq \{ (x,y,z)\in \Z_{\ge 1}^3  :  x^r + y^s = z^t,\, \gcd(x,y,z) = 1 \}
\end{align*}
for the finite set of positive primitive solutions to the generalized Fermat equation with signature $(r,s,t)$. 
Define
\begin{align*}
f(r,s,t) \defeq \sup\limits_{(x,y,z) \in S} \log(x^r y^s z^t) \in  \R_{\ge 0},
\end{align*}
where we shall write $f(r,s,t) = 0$ if $S(r,s,t)$ is an empty set.

Then the explicit upper bounds for $f(r,s,t)$ are presented in Table \ref{3-table: Fermat upper bound 1}.
In particular, we have $f(r,s,t) \le 573$ for $r,s,t\ge 8$;
$f(r,s,t) \le 907$ for $r,s,t\ge 5$;
$f(r,s,t) \le 2283$ for $r,s,t\ge 4$;
$f(r,s,t) \le 14750$ for $\min\{r,s\} \ge  4$ or $t\ge 4$;
and $f(r,s,t) \le 24626$ for $r,s,t\ge 3$.
\begin{table}[ht]
\centering
\caption{Upper bounds for $f(r,s,t)$}
\label{3-table: Fermat upper bound 1}
\small\begin{tabular}{|l|l|l|l|}
\hline
$(\min\{r,s\}, t)$ & Bounds &
$\min\{r,s,t\}$ & Bounds \\ 
\hline      
$(3,3)$ & $24626$ & $4$ & $2283$ \\
$(3,4)$ & $14750$ &  $5$& $907$ \\
$(3,n)$, $n\ge 5$ & $6648$  &  $6$& $697$ \\
$(4,3)$ & $7254$ & $7$& $635$ \\
$(n,3)$, $n\ge 5$ &  $3406$   & $\ge 8$& $573$ \\
\hline
\end{tabular}\normalsize
\end{table}

\end{thm}
\begin{proof}
Suppose that $S(r,s,t) \neq \emptyset$, and let $(x,y,z) \in S(r,s,t)$ be a positive primitive solution to $x^r+y^s=z^t$.
Write  $(a,b,c) \defeq (x^r, y^s, z^t)$, $h \defeq \log(x^ry^sz^t) = \log(abc)$,
then $a + b = c$ and $\gcd(a,b,c) = 1$.

We shall prove that
\begin{align} \label{4-3-eq1}
    \log\rad(abc) \le \log(xyz) < \frac{5}{16}\cdot h, \quad \text{for} \; h \ge \log(4).
\end{align}
Assume without loss of generality that $s\ge r$. Then since $(r,s,t)\neq (3,3,3)$, which is proved by Euler, we have $t \ge 4$; or $s\ge r\ge 4$; or $s \ge 4$, $r = t = 3$.

When $t \ge 4$, since $c > a,b$, we have 
$abc > a^{\frac{12}{11} }  b^{\frac{12}{11} }  c^{\frac{9}{11} } = x^{\frac{12r}{11} }  y^{\frac{12s}{11} }  z^{\frac{9t}{11} } \ge (xyz)^{\frac{36}{11} } $.
Hence $\log\rad(abc) \le \log(xyz) < \frac{11}{36}\log(abc) = \frac{11}{36}h <  \frac{5}{16} h.$

When $s \ge r \ge 4$, since $2ab \ge c$ by $a+b=c$, we have 
$2^{\frac{1}{7}}abc > a^{\frac{6}{7} }  b^{\frac{6}{7} }  c^{\frac{8}{7} } = x^{\frac{6r}{7} }  y^{\frac{6s}{7} }  z^{\frac{8t}{7} } \ge (xyz)^{\frac{24}{7} } $.
Hence when $h \ge \log(4)$, we have $\log\rad(abc) \le \log(xyz) < \frac{7}{24}\log(abc)+\frac{\log(2)}{24} = \frac{7}{24}h + \frac{\log(2)}{24} \le \frac{5}{16} h.$

When $s\ge 4$, $r = t = 3$, write $\lambda = \frac{s-3}{s+1} \ge \frac{1}{5}$ by $s\ge 4$.
Since $b = y^s > z^2 > xz$ by (\ref{4-1-eq3}), we have 
$ abc = x^3 y^s z^3 > x^{3+\lambda} y^{s-s\lambda} z^{3+\lambda} = (xyz)^{3+\lambda} \ge (xyz)^{\frac{16}{5}}$.
Hence $\log\rad(abc) \le \log(xyz) < \frac{5}{16}\log(abc) = \frac{5}{16}h .$
This proves (\ref{4-3-eq1}).

Assume that $h \ge \log(4)$, then by Theorem \ref{3-thm:ABC-type inequality}, (i) and  (\ref{4-3-eq1}),
we have
\begin{align*}
    h \le 3\log\rad(abc) + 9\sqrt{h\cdot \log(h)}
    \le \frac{15}{16}\cdot h + 9\sqrt{h\cdot \log(h)},
\end{align*}
which is impossible for $h \ge 10^6$.
Hence we must have $h < 10^6$.
The computation in \cite{code_of_zpzhou} based on Lemma \ref{4-lem:upper bound I} shows that if $h < 10^6$, then $h$ must be less than the upper bounds presented in Table \ref{3-table: Fermat upper bound 1}.
Hence the theorem is proven.
\end{proof}

By making use of the ``entirely elementary estimate'' for positive solutions of the Fermat equation presented in \cite{ExpEst}, Lemma 5.7,
we shall provide an alternative approach to prove the ``Fermat's Last Theorem'' for prime exponents $\ge 11$.

\begin{cor} \label{4-cor:FLT}
Let $$p\ge 11$$ be a prime number. 
Then there does not exist any triple $(x, y, z)$ of positive integers that satisfies the Fermat equation
$$x^p + y^p = z^p.$$
\end{cor}
\begin{proof}
By dividing $(x,y,z)$ by their greatest common divisor, we shall assume that $(x,y,z)$ is a triple of  positive coprime integers such that $x^p + y^p = z^p$ and $x < y < z$.

Since
\begin{align*}
x^p = z^p - y^p \ge z^p - (z-1)^p > z^{p-1}, 
\end{align*}
we have 
\begin{align} \label{4-4-eq1}
x^p  y^p z^p > z^{p-1} \cdot (z^p - z^{p-1}) \cdot z^p > (z-1)^{3p-1}.
\end{align}
Since $p\ge 11$, by Theorem \ref{4-thm:conclusions about the generalized Fermat equation}, we have  $\log(x^p y^p z^p) \le 600$.
Then by (\ref{4-4-eq1}), we have
\begin{align} \label{4-4-eq2}
\log(z-1) < \frac{1}{3p-1} \cdot \log(x^p y^p z^p) \le \frac{600}{3p-1} < 20. 
\end{align}

However, by the estimate in \cite{ExpEst}, Lemma 5.7,
we have $z > \frac{(p+1)^p}{2}$. Hence for $p\ge 11$, we have 
\begin{align} \label{4-4-eq3}
\log(z-1) \ge \log(\frac{(p+1)^p}{2}) = p \cdot \log(p+1) - \log(2) > 26.
\end{align}
--- a contradiction between (\ref{4-4-eq2}) and (\ref{4-4-eq3}).
\end{proof}

\begin{tiny-remark}
(i) To prove Fermat's Last Theorem (FLT), we only need to prove that there does not exist any triple $(x, y, z)$ of positive integers that satisfies the Fermat equation $$x^n + y^n = z^n, $$
where $n=4$ or $n=p$ is a prime number $\ge 3$.

(ii) The case $n=4$ is proved by Fermat using the technique of infinite descent;
the case $ n=3$ is proved by Euler;
the case $n=5$ is proved by Dirichlet and Legendre around 1825;
the case $n = 7$ is proved by Gabriel Lamé in 1839. 
All of these cases have elementary proofs.

(iii) in Corollary \ref{4-cor:FLT}, by applying Theorem \ref{4-thm:conclusions about the generalized Fermat equation} and the elementary estimate found in \cite{ExpEst}, Lemma 5.7, we have proven FLT for prime exponents $\ge 11$.
Hence we conclude that the results of the present paper,
combined with the historical proofs for the cases $n = 3, 4, 5, 7$ montioned in (ii),
yield an unconditional new alternative proof [i.e., to the proof of \cite{Wiles}] of Fermat's Last Theorem.
\end{tiny-remark}

\begin{cor} \label{4-cor:generalized Fermat equation}
Let $r,s,t$ be positive integers such that $r,s,t \ge 20$, or $(r,s,t)$ is a permutation of $(3,3,n)$ for $n\ge 3$.
Then there does not exist any triple $(x, y, z)$ of non-zero coprime integers that satisfies the generalized Fermat equation
$$x^r + y^s = z^t.$$
\end{cor}
\begin{proof}
Suppose that $(x, y, z)$ is a non-trivial primitive solution to the generalized Fermat equation $x^r + y^s = z^t.$ Then by permuting $(r,s,t)$, we can assume that $x,y,z \ge 1$.
Then since $r,s,t \ge 3$, we have $x,y,z \ge 2$ by Catalan's conjecture.

(1) When $(r,s,t)$ is a permutation of $(3,3,n), n\ge 3$, we have $n > 10^9$ by \cite{Chen-Siksek}, also cf. Theorem \ref{4-thm:conclusions about the generalized Fermat equation}.
However, by Theorem \ref{4-thm-bounds} we have $\log(2^n) < \log(x^r y^s z^t) \le 24626$, hence $n < 24626 / \log(2) < 10^5$ -- a contradiction!

(2) When $r,s,t \ge 20$,  by Theorem \ref{4-thm-bounds} we have $\log(x^r y^s z^t) < 600$.
However, the computation in \cite{code_of_zpzhou} shows that for any $r, s, t \ge 20$, there does not exist any triple $(x, y, z)$ of positive coprime integers such that $x^r + y^s = z^t$ and $\log(x^r y^s z^t) < 600$ --- a contradiction!

The corollary follows from (1) and (2).
\end{proof}

\begin{tiny-remark}
In the subsequent paper, we will introduce new insights to further explore the generalized Fermat equations. We will establish new upper bounds for non-trivial primitive solutions and demonstrate the non-existence of such solution for an expanded set of signatures.
\end{tiny-remark}

\bibliographystyle{plain}
\bibliography{zpzhou.bib}

\Addresses

\end{document}